\documentclass[11pt,a4paper] {article}
\usepackage{latexsym,amsmath,enumerate,amssymb,amsbsy,amsthm}
 
\usepackage{enumerate,verbatim,tikz,upgreek, tipa} 

\newtheorem{theorem}{Theorem}[section]
\newtheorem{corollary}[theorem]{Corollary}
\newtheorem{lemma}[theorem]{Lemma}
\newtheorem{proposition}[theorem]{Proposition}
\theoremstyle{definition}

\newtheorem{remark}[theorem]{Remark}

\newcommand{\I}{\textnormal{I}}
\newcommand{\II}{\textnormal{I\hspace*{-0.4ex}I}}
\newcommand{\III}{\textnormal{I\hspace*{-0.4ex}I\hspace*{-0.4ex}I}}

\newcommand{\Ip}{\textnormal{I\hspace*{-0.3ex}}^\prime}

\newcommand{\IIp}{\textnormal{I\hspace*{-0.4ex}I\hspace*{-0.3ex}}^\prime}
\newcommand{\IIIp}{\textnormal{I\hspace*{-0.4ex}I\hspace*{-0.4ex}I\hspace*{-0.3ex}}^\prime}

\renewcommand\footnotemark{}

\date{}

\begin{document}
    
    \title{{Self-Dual  and Complementary Dual Abelian Codes over Galois Rings}
    }
    
    \author{Somphong Jitman  and San Ling}

    \thanks{S. Jitman (Corresponding Author) is with the Department of Mathematics, Faculty of Science, Silpakorn University,
        Nakhon Pathom 73000, THAILAND (sjitman@gmail.com).}

    \thanks{S. Ling is with the Division of
        Mathematical Sciences,
        School of Physical and Mathematical Sciences, Nanyang Technological
        University,    Singapore 637371, 
        REPUBLIC of SINGAPORE (lingsan@ntu.edu.sg).}
 
    \maketitle
    
    \begin{abstract}  
    Self-dual and complementary dual cyclic/abelian  codes over finite fields form   important classes of linear codes that have been extensively  studied due to their rich algebraic structures and wide applications.  
    In this paper,   abelian codes over Galois rings are studied in terms  of the ideals  in  the  group ring ${\rm GR}(p^r,s)[G]$, where $G$ is a finite abelian group and ${\rm GR}(p^r,s)$ is a Galois ring.    Characterizations of     self-dual abelian  codes have been  given together with     necessary and sufficient conditions for the existence of    a   self-dual abelian code in ${\rm GR}(p^r,s)[G]$. A general    formula  for the number of  such self-dual codes  is established.    In the    case  where    $\gcd(|G|,p)=1$,   the number of    self-dual  abelian codes in ${\rm GR}(p^r,s)[G]$ is completely and explicitly  determined. 
    Applying  known results on cyclic codes of length $p^a$ over ${\rm GR}(p^2,s)$,   an  explicit formula for the number  of    self-dual  abelian codes in ${\rm GR}(p^2,s)[G]$ are  given, where the Sylow $p$-subgroup of $G$ is cyclic. 
     Subsequently, the characterization and  enumeration of complementary dual abelian codes   in ${\rm GR}(p^r,s)[G]$  are established. 
    The analogous results for    self-dual and complementary dual  cyclic  codes  over Galois rings are therefore obtained as  corollaries.

        \noindent \textbf{Keywords.} Abelian codes,  Galois rings,  Self-dual codes, Complementary dual codes,  Codes over rings
        \\
        \textbf{2010 AMS Classification.}  94B15, 94B05, 16A26
    \end{abstract}


\section{Introduction}

Algebraically  structured codes over finite fields with self-duality  and complementary duality are important families of linear codes  that have been extensively studied for both  theoretical and practical reasons  (see \cite{BGG2012},  \cite{BJU2018}, \cite{JLX2011}, \cite{JLLX2012}, \cite{JLS2012},  \cite{NRS2006}, \cite{W2002}, \cite{YM1994},  and references therein). Codes over finite rings have been interesting  since   it was proven that some binary non-linear codes such as the Kerdock, Preparata, and Goethal codes are the Gray images of linear codes over $\mathbb{Z}_4$ \cite{HKCSS1994}. Algebraically  structured codes such as cyclic, constacyclic, and abelian codes  have  extensively been studied over $\mathbb{Z}_{p^r}$, Galois rings, and finite chain rings in general (see \cite{DL2004},\cite{KR2003}, and references therein). 

The  characterization and  enumeration of Euclidean self-dual cyclic codes over finite fields have been established in \cite{JLX2011} and generalized to Euclidean and Hermitian self-dual abelian  codes over finite fields  in \cite{JLLX2012} and \cite{JLS2012}, respectively. Over some finite rings, a characterization of self-dual cyclic, constacyclic  and abelian codes  has been done (see, for example, \cite{BGG2012}, \cite{DL2004},\cite{KLL2008}, \cite{KLL2012}, \cite{SE2009}, and  \cite{W2002}). In \cite{BGG2012}, \cite{BSG2016}, \cite{BJU2018} and \cite{S2006},  characterization and    enumeration of  Euclidean and Hermitian self-dual    cyclic codes over  finite chain rings have been discussed.    Euclidean complementary dual cyclic codes over finite fields have been studied  in \cite{YM1994}. Recently, they have been generalized to Euclidean and Hermitian  complementary dual  abelian  codes over finite fields  in  \cite{BJU2018}. The complete characterization and enumeration of complementary dual abelian codes over finite fields have been established in the said paper.

In this paper, we focus on  abelian  codes over Galois rings ${\rm GR}(p^r,s)$, {\it i.e.,} ideals in the group ring ${\rm GR}(p^r,s)[G]$ of an abelian group $G$ over a Galios ring ${\rm GR}(p^r,s)$.  Specifically, we study  self-dual  and complementary dual abelian codes   in  ${\rm GR}(p^r,s)[G]$ with respect to the Euclidean and Hermitian inner products.  We      characterize  such self-dual abelian codes  and determine  necessary and sufficient conditions for the existence of  a   self-dual abelian code in  ${\rm GR}(p^r,s)[G]$. We give  a formula  for      the number of    self-dual abelian  codes in ${\rm GR}(p^r,s)[G]$.      Under the restriction   $i)$  $\gcd(|G|,p)=1$;   or  $ii)$    $r=2$ {and} the Sylow $p$-subgroup of $G$ is cyclic,  the numbers of  self-dual  abelian codes in ${\rm GR}(p^r,s)[G]$ are explicitly determined.    Subsequently, the characterization and  enumeration of complementary dual abelian codes   in ${\rm GR}(p^r,s)[G]$  are given. The number of  complementary dual abelian codes   in ${\rm GR}(p^r,s)[G]$ is shown to be independent of $r$ and the Sylow $p$-subgroup of $G$.
  
We note that the Hermitian duality is meaningful only   when $s$ is even.
Since we study Euclidean and Hermitian self-dual  codes in parallel, the assumption that $s$ is even is included whenever we refer to the Hermitian duality.

The paper is organized as follows. In Section~\ref{sec2}, we recall and prove some basic results for  group rings,  abelian codes, and  their duals. In Section~\ref{sec4}, we present the characterization and a general set up for the   enumeration of   self-dual abelian codes  in ${\rm GR}(p^r,s)[G]$.   The  complete enumeration of  Euclidean  and Hermitian self-dual  abelian codes in ${\rm GR}(p^r,s)[G]$ is given in the special cases where $i)$  $\gcd(p,|G|)=1$;  and $ii)$  $r=2$  and the Sylow $p$-subgroup of $G$ is cyclic.    In  Section~\ref{sec4p},   the characterization and  enumeration of complementary dual abelian codes   in ${\rm GR}(p^r,s)[G]$  are given.

\section{Preliminaries}\label{sec2}
In this section, we  recall some definitions and basic properties of abelian codes  and   prove some   results on their Euclidean and Hermitian duals.

\subsection{Abelian Codes}
For a finite commutative ring $R$ with identity and a finite abelian group $G$, written additively,  let  $R[G]$ denote the {\it group ring} of $G$ over~$R$. The elements in $ R[G]$ will be written as $\sum_{g\in G}\alpha_{{g }}Y^g $,
where $ \alpha_{g }\in R$.  The addition and the multiplication in $ R[G]$ are  given as in the usual polynomial rings over $R$ with the indeterminate $Y$, where the indices are computed additively in $G$. By convention, $Y^0=1$ is the identity of  $R$, where $0$ is the additive identity of $G$.   

An  {\em abelian code} in $R[G]$ is  defined to be an ideal of   $R[G]$.  If  $G$ is cyclic, this code becomes a  cyclic code. For this case,  an abelian   code will be referred to as a {\em cyclic code}.     It is well known that cyclic codes of length $n$ over $R$ can also be regarded as ideals in the quotient polynomial ring $R[X]/\langle X^n-1\rangle$.

From now on, we focus on the case where the ring is a Galois ring ${\rm GR}(p^r,s)$, a Galois extension of degree $s$ of an integer residue ring $\mathbb{Z}_{p^r}$.   
Let $\xi$ be an element in ${\rm GR}(p^r,s)$ that  generates a Teichm\"{u}ller set $\mathcal{T}_s$ of  ${\rm GR}(p^r,s)$. In other words, $\mathcal{T}_s=\{0,1,\xi,\xi^2,\dots,\xi^{p^{s}-2}\}$.
Then every  element  in  ${\rm GR}(p^r,s)$   has a unique $p$-adic expansion of the form
\[\alpha=\alpha_0+\alpha_1p+\dots+\alpha_{r-1}p^{r-1},\]
where $\alpha_i\in \mathcal{T}_s$ for all $i=0,1,\dots,r-1$.  If $s$ is even,  let $\bar{~}$  denote the   automorphism 
on   ${\rm GR}(p^r,s)$ defined by
\begin{align}\label{aut}
	\overline{\alpha}=\alpha_0^{p^{ {s}/{2}}}+\alpha_1^{p^{ {s}/{2}}}p+\dots+\alpha_{r-1}^{p^{{s}/{2}}}p^{r-1}.
	\end{align} 
For more details  on Galois rings, we refer the readers to \cite{W2003}.  

Assume that   $G\cong A\times P$, where $P$ is the Sylow $p$-subgroup of $G$ and $A$ is its  complementary subgroup of $P$ in $G$. Let $\mathcal{R}:= {\rm GR}(p^r,s)[A]$. 
Then the map $\Phi: {\rm GR}(p^r,s)[G] \to \mathcal{R}[P]$ given by
\begin{align*}
	\Phi(\sum_{ a\in   A}\sum_{b\in P}\alpha_{a+b}Y^ {a+b})= \sum_{b\in   P}\boldsymbol{\alpha}_b(Y)Y^b, 
\end{align*}
where 
$
\boldsymbol{\alpha}_b(Y) =	 \sum_{a\in A}\alpha_{a+b}Y^a\in \mathcal{R},
$
 is a well-known ring isomorphism.   
\begin{lemma}\label{lemma2.1}
	The map $\Phi$ induces a one-to-one correspondence between abelian codes  in ${\rm GR}(p^r,s)[G]$ and  abelian codes  in $\mathcal{R}[P]$.  
\end{lemma}

An abelian  code $\mathcal{C}$ in  ${\rm GR}(p^r,s)[G]$ is said to be {\em Euclidean self-dual} (resp., {\em Euclidean complementary dual}) if  $\mathcal{C}=\mathcal{C}^{\perp_{\rm E}}$ (resp., $\mathcal{C} \cap \mathcal{C}^{\perp_{\rm E}}=\{0\}$), where $\mathcal{C}^{\perp_{\rm E}}$ is the dual of $\mathcal{C}$ with respect to  the form
\begin{align*}
	\langle \boldsymbol{u},\boldsymbol{v}\rangle_{\rm E}:=\sum_{g\in   G} \alpha_{g}{\beta_{g}},
\end{align*}
where  $ \boldsymbol{u}=\sum_{g\in   G} \alpha_{g}Y^{g}$ and $ \boldsymbol{v}=\sum_{g\in   G} \beta_{g}Y^{g}$.

Define an {\em involution} ~$\widehat{~}$~  on $\mathcal{R}$ to be the ​${\rm GR}(p^r,s)$-module homomorphism that fixes ${\rm GR}(p^r,s)$ and   sends $Y^a$ to $Y^{-a}$ for all $a\in A$.  An abelian  code $D$ in  $\mathcal{R}[P]$ is said to be {\em ~$\widehat{~}$-self-dual} if  $D=D^{\perp\displaystyle\hspace*{-0.4ex}{~ \atop \widehat{~} }}$, where $D^{\perp\displaystyle\hspace*{-0.4ex}{~ \atop \widehat{~} }}$ is the dual of $D$ with respect to  the form 
\begin{align*}
    { \langle \boldsymbol{x},\boldsymbol{y}\rangle \displaystyle\hspace*{-0.4ex}{~ \atop \widehat{~} } 
:=\sum_{b\in   P}\boldsymbol{x}_b(Y)\widehat{\boldsymbol{y}_b(Y)},}
\end{align*}
where $\boldsymbol{x}=\sum_{b\in   P}\boldsymbol{x}_b(Y)Y^b$
 and 
$\boldsymbol{y}=\sum_{b\in   P}\boldsymbol{y}_b(Y)Y^b$.

In addition, if $s$ is even,  an abelian code $\mathcal{C}$ in ${\rm GR}(p^r,s)[G]$ is said to be {\em Hermitian self-dual}  (resp., {\em Hermitian complementary dual}) if  $\mathcal{C}=\mathcal{C}^{\perp_{\rm H}}$ (resp., $\mathcal{C}\cap \mathcal{C}^{\perp_{\rm H}}=\{0\}$), where $\mathcal{C}^{\perp_{\rm H}}$ is the dual of $\mathcal{C}$ with respect to  the form
\begin{align*}
	\langle \boldsymbol{u},\boldsymbol{v}\rangle_{\rm H}:=\sum_{g\in   G} \alpha_{g}\overline{\beta_{g}}, 
\end{align*}
where  $ \boldsymbol{u}=\sum_{g\in   G} \alpha_{g}Y^{g}$ and $ \boldsymbol{v}=\sum_{g\in   G} \beta_{g}Y^{g}$.

Define an {\em involution}~$\widetilde{~}$~on $\mathcal{R}$ to be the ​${\rm GR}(p^r,s)$-module homomorphism   that sends $\alpha$ to $\overline{\alpha}$ for all $\alpha \in {\rm GR}(p^r,s)$   and   sends $Y^a$ to $Y^{-a}$ for all $a\in A$. An abelian  code $D$ in  $\mathcal{R}[P]$ is said to be {\em $\sim$-self-dual} if  $D=D^{\perp_{\sim}}$, where $D^{\perp_{\sim}}$ is the dual of $D$ with respect to  the form 
\begin{align*}
	\langle \boldsymbol{x},\boldsymbol{y}\rangle_{\sim}:=\sum_{b\in   P}\boldsymbol{x}_b(Y)\widetilde{\boldsymbol{y}_b(Y)},
\end{align*}
where $\boldsymbol{x}=\sum_{b\in   P}\boldsymbol{x}_b(Y)Y^b$
 and 
$\boldsymbol{y}=\sum_{b\in   P}\boldsymbol{y}_b(Y)Y^b$.

Similar to the finite field case,  the following relations among the above forms    can be verified using  arguments similar to those in  \cite[Proposition 2.4]{JLLX2012} and \cite[Proposition 2.4]{JLS2012}.

\begin{lemma}
Let  $\boldsymbol{u}, \boldsymbol{v}\in {\rm GR}(p^r,s)[G]$.  Then the following statements hold.
\begin{enumerate}[$i)$]
\item  $\langle Y^{g} \boldsymbol{u}, \boldsymbol{v}\rangle_{\rm E}=0$ for all $g\in G$ if and only if $\langle Y^b\Phi(\boldsymbol{u}), \Phi(\boldsymbol{v})\rangle\displaystyle\hspace*{-0.4ex}{~ \atop \widehat{~} }=0$ for all $b\in P$.
\item   $\langle Y^{g} \boldsymbol{u}, \boldsymbol{v}\rangle_{\rm H}=0$ for all $g\in G$ if and only if $\langle Y^b\Phi(\boldsymbol{u}), \Phi(\boldsymbol{v})\rangle_{\sim}=0$ for all $b\in P$.
\end{enumerate}
\end{lemma}
The next corollary follows immediately. 
\begin{corollary}
Let $\mathcal{C}$ be an abelian code in ${\rm GR}(p^r,s)[G]$. Then the following statements hold.
\begin{enumerate}[$i)$]
\item  $\Phi(\mathcal{C})^{\perp\displaystyle\hspace*{-0.4ex}{~ \atop \widehat{~} }}=\Phi(\mathcal{C}^{\perp_{\rm E}})$. In particular, $\mathcal{C}$  is Euclidean self-dual  if and only if  $\Phi(\mathcal{C})$  is $\widehat{~~}$-self-dual.
\item  $\Phi(\mathcal{C})^{\perp_{\sim}}=\Phi(\mathcal{C}^{\perp_{\rm H}})$. In particular, $\mathcal{C}$  is Hermitian self-dual  if and only if  $\Phi(\mathcal{C})$  is $\sim$-self-dual.
\end{enumerate}
\end{corollary}

Therefore, to study Euclidean and Hermitian self-dual abelian codes in   ${\rm GR}(p^r,s)[G]$, it is sufficient to consider  $\widehat{~~}$-self-dual and $\sim$-self-dual  abelian codes in   $\mathcal{R}[P]$, respectively.

\subsection{Decomposition and Dualities}
Recall that $p$ represents  a prime number, $s$ is a positive integer, and $A$ is a finite abelian group such that  $\gcd(p,|A|)=1$. 

For coprime positive integers $i,j$, let ${\rm ord}_i(j)$ denote the multiplicative order of $j$ modulo $i$. For\,$a\in A$,\,denote\,by\,${\rm ord}(a)$\,the\,additive\,order\,of\,$a$\,in\,$A$. For each $a\in A$,   a {\it $p^s$-cyclotomic class}   of $A$ containing $a$  is defined to be the set
	\begin{align*}
		S_{p^s}(a):=&\{{p^{si}}\cdot a \mid i=0,1,\dots\}=\{{p^{si}}\cdot a \mid 0\leq i< {\rm ord}_{{\rm ord}(a)}(p^s) \}, 
	\end{align*}
	where $p^{si}\cdot a:= \sum_{j=1}^{p^{si}}a$ in $A$.   
A $p^{s}$-cyclotomic class $S_{p^s}(a)$ is said to be of  {\em type} $\I$    if  $a=-a$, {\em type} $\II$  if $S_{p^s}(a)=S_{p^s}(-a)$ and $a\neq -a$, or {\em type} $\III$   if $S_{p^s}(-a)\neq S_{p^s}(a)$.
If $s$ is even, a $p^{s}$-cyclotomic class $S_{p^s}(a)$ is said to be of  {\em type} $\IIp$    if $S_{p^s}(a)=S_{p^s}(-p^{ {s}/{2}}\cdot a)$ or {\em type} $\IIIp$   if $S_{p^s}(-p^{{s}/{2}}\cdot a)\neq S_q(a)$, where $-p^{{s}/{2}}\cdot a $ denotes $p^{{s}/{2}}\cdot (-a)$.

\begin{remark}\label{rem1} We have the following facts for the $p^s$-cyclotomic classes of $A$ (see \cite[Remark 2.5]{JLLX2012} and \cite[Remark 2.6]{JLS2012}).
	\begin{enumerate}
	    \item A $p^s$-cyclotomic class of   type $\I$ has cardinality one. 
		\item $S_{p^s}(0)$ is a $p^s$-cyclotomic class of both types $\I$ and $\IIp$. 
		\item If a $p^s$-cyclotomic class of type $\II$  exists, then its    cardinality is even.  Moreover, if $S_{p^s}(a)$ is a $p^s$-cyclotomic class of   type $\II$  of cardinality $2\nu$, then $-a=p^{s\nu}\cdot a$.
		\item A $p^s$-cyclotomic class of $A$ of type $\IIp$ has odd    cardinality. Moreover,  if $S_{p^s}(a)$ is a $p^s$-cyclotomic class of   type $\IIp$  of cardinality $\nu$, then $-a= p^{{s\nu}/{2}}\cdot a$ and  $-p^{{s}/{2}}\cdot a=p^{{s(\nu+1)}/{2}}\cdot a$. 
	\end{enumerate}
\end{remark}

Assume that $A$ has cardinality $m$ and exponent $M$. By the Fundamental Theorem of finite abelian groups, $A$ can be written  as a direct product  of finite cyclic groups
		$
			A=\prod_{i=1}^N\mathbb{Z}_{m_i},
		$
		where $\mathbb{Z}_{m_i}=\{0,1,\dots,m_i-1\}$ denotes  the additive cyclic group of order $m_i\geq 2 $ for all $1\leq i\leq N$. Then an element    $b\in A$ can be written as $b=(b_1,b_2,\dots,b_N)$, where $b_i\in \mathbb{Z}_{m_i}$. For each $h\in A$, let $\gamma_h :A \to \mathbb{Z}$
	  be defined by
		\begin{align}\label{formula}
			\gamma_h(b)={\sum_{i=1}^Nb_ih_i(M/m_i)},
		\end{align}
		where  the sum is a rational sum.  
 
Let $\mu$  be the order of $p^s$ modulo $M$. Denote by $\zeta$ a primitive $M$th root of unity in ${\rm GR}(p^r,s\mu)$.
 	For a given $\boldsymbol{c}=\sum_{a\in A} c_{a}Y^a\in \mathcal{R}:={\rm GR}(p^r,s)[A]$, its 
	Discrete Fourier Transform (DFT) is  $\Breve{\boldsymbol{c}}=\sum_{h\in A} \Breve{c}_hY^h$, 
	where
	\begin{align}\label{ca}
		\Breve{c}_h=\sum_{a\in A} {c}_{a} \zeta^{\gamma_h(a)}\in {\rm GR}(p^r,s\mu).
	\end{align}
Moreover, if $S_{p^s}(h)$ has cardinality $\nu$, then it is not difficult to verify that $\Breve{c}_h$ is contained in a subring of $ {\rm GR}(p^r,s\mu)$ which is isomorphic to  ${\rm GR}(p^r,s\nu)$.
 
Using this DFT, the decomposition of $\mathcal{R}:={\rm GR}(p^r,s)[A]$, where  $\gcd(p, |A|)=1$, has been given in  \cite{KR2003} in terms of the mix-radix representation of the elements in $A$.   In order to utilize   the decomposition in \cite{KR2003}  for characterizing self-dual codes,  we need to consider a suitable  rearrangement   of the terms in   the decomposition.

\subsubsection{Euclidean Case}
For the Euclidean self-duality,  we consider the rearrangement based on  the   $p^s$-cyclotomic classes of types $\I-\III$ as follows. Assume that $A$ contains $L$ $p^s$-cyclotomic classes. 
Without loss of generality, let $\{a_1, a_2, \dots, a_L\}$ be a set of representatives of the $p^s$-cyclotomic classes such that $\{a_i\mid i=1,2,\dots,t_{\I}\}$, $\{a_{t_{\I}+j} \mid j=1,2,\dots,t_{\II}\}$  and $\{a_{t_{\I}+t_{\II}+k}, a_{t_{\I}+t_{\II}+t_{\III}+k}=-a_{t_{\I}+t_{\II}+k} \mid k=1,2,\dots, t_{\III}\}$ are  sets of representatives of $p^s$-cyclotomic classes of types $\I, \II$, and $\III$, respectively, where $L=t_{\I}+t_{\II}+2t_{\III}$.  From the definition, $|S_{p^s}(a_i)|=1$ for all $i=1,2,\dots,t_{\I}$. From Remark \ref{rem1}, the order of the  $p^s$-cyclotomic classes of type $\II$ is even order. For $j=1,2,\dots,t_{\II}$, let  $2e_j$ denote the cardinality of $S_{p^s}(a_{t_{\I}+j})$.  For $k=1,2,\dots, t_{\III}$, $S_{p^s}(a_{t_{\I}+t_{\II}+k})$ and $S_{p^s}(a_{t_{\I}+t_{\II}+t_{\III}+k})$  have the same cardinality and   denote it by $f_k$.

Rearranging the terms in the decomposition    in  \cite{KR2003} based on the   $p^s$-cyclotomic classes of $A$ of types $\I-\III$,  we have 

\begin{align}\label{E1}
	\mathcal{R}\cong \left(\prod_{i=1}^{t_{\I}}{\rm GR}(p^r,s)\right)
	                                 &\times \left(\prod_{j=1}^{t_{\II}}{\rm GR}(p^r,2se_j)\right) \times \left(\prod_{k=1}^{t_{\III}}\left({\rm GR}(p^r,sf_k)\times {\rm GR}(p^r,sf_k)\right)\right), \tag{E1}
\end{align}
   where  
  ${\rm GR}(p^r,2se_j)$ is induced by $S_{p^s}(a_{t_{\I}+j})$   for all 
  $j=1,2,\dots,t_{\II}$  and   ${\rm GR}(p^r,sf_k)\times {\rm GR}(p^r,sf_k)$ is induced by $(S_{p^s}(a_{t_{\I}+t_{\II}+k}), $ $S_{p^s}(-a_{t_{\I}+t_{\II}+k}))$  for all $k=1,2,\dots,t_{\III}$.
For more details   and the explicit isomorphism, the readers may refer to \cite[Section II]{KR2003}.

It follows that 
\begin{align}\label{E2}
	\mathcal{R}[P]\cong &\left(\prod_{i=1}^{t_{\I}}{\rm GR}(p^r,s)[P]\right)
	                                	 \times \left(\prod_{j=1}^{t_{\II}}{\rm GR}(p^r,2se_j)[P]\right)\times \left(\prod_{k=1}^{t_{\III}}\left({\rm GR}(p^r,sf_k)[P]\times {\rm GR}(p^r,sf_k)[P]\right)\right). \tag{E2}
\end{align}
Therefore, by Lemma \ref{lemma2.1}, every abelian code in ${\rm GR}(p^r,s)[G]\cong \mathcal{R}[P]$ can be written in the form
\begin{align}\label{E3}
	\mathcal{C}\cong \left(\prod_{i=1}^{t_{\I}}U_i \right)\times \left(\prod_{j=1}^{t_{\II}}V_j \right) \times \left(\prod_{k=1}^{t_{\III}}(W_k\times W_k^\prime) \right), \tag{E3}
\end{align}
where $U_i$ is an abelian code in   ${\rm GR}(p^r,s)[P]$, $V_j$ is an abelian code in   ${\rm GR}(p^r,2se_j)[P]$, and $W_k,W_k^\prime$ are  abelian codes in  ${\rm GR}(p^r,sf_k)[P]$ for all $i=1,2,\dots, t_{\I}$, $j=1,2,\dots,t_{\II}$, and   $k=1,2,\dots,t_{\III}$.

The Euclidean dual of $\mathcal{C}$ in (\ref{E3}) can be viewed to be of the form
\begin{align} \label{dualE}
	\mathcal{C}^{\perp_{\rm E}}\cong \left(\prod_{i=1}^{t_{\I}} U_i^{\perp_{\rm E}}  \right)&\times \left(\prod_{j=1}^{t_{\II}} V_j ^{\perp_{\rm H}} \right)\times \left(\prod_{k=1}^{t_{\III}} \left( (W_k^\prime) ^{\perp_{\rm E}}\times  W_k^{\perp_{\rm E}}\right) \right).\tag{E4}
\end{align}
The detailed justification for  (\ref{dualE}) is provided    in Appendix \ref{SSAEC}.

 \subsubsection{Hermitian Case}
In the case where $s$ is even, we consider the other rearrangement of the  decomposition of $\mathcal{R}$ in terms of the   $p^s$-cyclotomic classes of $A$ of types $\IIp$ and $\IIIp$. Let   $\{b_1=0, b_2, \dots, b_L\}$ denote
 a set of representatives of the $p^s$-cyclotomic classes such that $\{b_j \mid j=1,2,\dots,t_{\IIp}\}$ and 
 $\{b_{t_{\IIp}+k}, b_{t_{\IIp}+t_{\IIIp}+k}=-p^{{s}/{2}}\cdot b_{t_{\IIp}+k} \mid k=1,2,\dots, t_{\IIIp} \}$     represent   $p^s$-cyclotomic classes of types $\IIp$ and $\IIIp$, respectively, where $L=t_{\IIp}+2t_{\IIIp} $.  
For $j=1,2,\dots,t_{\IIp}$, let  $\acute{e}_j$ denote the cardinality of $S_{p^s}(b_{j})$. 
 For $k=1,2,\dots, t_{\IIIp}$, $S_{p^s}(b_{t_{\IIp}+k})$ and $S_{p^s}(b_{t_{\IIp}+t_{\IIIp}+k})$  have the same cardinality and   denote it  by $\acute{f}_k$.

Rearranging the terms in the decomposition    in  \cite{KR2003} based on the   $p^s$-cyclotomic classes of $A$ of types $\IIp$ and $\IIIp$, we have 
 
\begin{align}\label{H1}
	\mathcal{R}\cong  \left(\prod_{j=1}^{t_{\IIp}}{\rm GR}(p^r,s\acute{e}_j)\right)
	\times \left(\prod_{k=1}^{t_{\IIIp}}\left({\rm GR}(p^r,s\acute{f}_k)\times {\rm GR}(p^r,s\acute{f}_k)\right)\right), \tag{H1}
\end{align}
    where   ${\rm GR}(p^r,s\acute{e}_j)$ is induced by $S_{p^s}(b_{j})$  for all $j=1,2,\dots,t_{\IIp}$     and   \  ${\rm GR}(p^r,s\acute{f}_k)\times {\rm GR}(p^r,s\acute{f}_k)$ is induced by $\left(S_{p^s}(b_{t_{\IIp}+k}), S_{p^s}(-p^{{s}/{2}}\cdot b_{t_{\IIp}+k})\right)$ for all $k=1,2,\dots,t_{\IIIp}$.

Consequently,   
\begin{align}\label{H2}
	\mathcal{R}[P]\cong& \left(\prod_{j=1}^{t_{\IIp}}{\rm GR}(p^r,s\acute{e}_j)[P]\right)  \times \left(\prod_{k=1}^{t_{\IIIp}}\left({\rm GR}(p^r,s\acute{f}_k)[P]\times {\rm GR}(p^r,s\acute{f}_k)[P]\right)\right), \tag{H2}
\end{align}
and, by Lemma \ref{lemma2.1}, every abelian code in ${\rm GR}(p^r,s)[G] \cong \mathcal{R}[P]$ can be viewed as
\begin{align}\label{H3}
	\mathcal{C}\cong  \left(\prod_{j=1}^{t_{\IIp}}E_j \right) \times \left(\prod_{k=1}^{t_{\IIIp}}(F_k\times F_k^\prime) \right), \tag{H3}
\end{align}
where $E_j$ is an abelian code   in   ${\rm GR}(p^r,s\acute{e}_j)[P]$ and $F_k,F_k^\prime$ are  abelian codes in  ${\rm GR}(p^r,s\acute{f}_k)[P]$ for all $j=1,2,\dots, t_{\IIp}$ and $k=1,2,\dots,t_{\IIIp}$.

Then the Hermitian dual of $\mathcal{C}$ in (\ref{H3}) ้has the form
\begin{align} \label{dualH}
	\mathcal{C}^{\perp_{\rm H}}\cong  \left(\prod_{j=1}^{t_{\IIp}} E_j ^{\perp_{\rm H}} \right)\times \left(\prod_{k=1}^{t_{\IIIp}} \left( (F_k^\prime) ^{\perp_{\rm E}}\times  F_k^{\perp_{\rm E}}\right) \right).\tag{H4}
\end{align}
The detailed discussion for  (\ref{dualH})  is provided  in Appendix \ref{SSAHC}.

\section{Self-Dual Abelian   Codes in ${\rm GR}(p^r,s)[G]$} \label{sec4}
In this section, we characterize and enumerate Euclidean and Hermitian self-dual abelian codes in ${\rm GR}(p^r,s)[G]$. We determine necessary and sufficient conditions for the existence of  self-dual abelian codes in   ${\rm GR}(p^r,s)[G]$  in Subsection \ref{subsec4.1} and  followed  by general results for the enumeration of such self-dual codes in Subsection \ref{subsec4.2}.  Some special  cases will be discussed in Subsections \ref{subsec4.3} and \ref{subsec4.4}.

\subsection{The Existence of Self-Dual Abelian Codes} \label{subsec4.1}
The characterizations of Euclidean and Hermitian self-dual abelian codes in ${\rm GR}(p^r,s)[G]$ are given as follows.

From (\ref{E3}) and (\ref{dualE}),  the characterization of Euclidean   self-dual abelian codes in ${\rm GR}(p^r,s)[G]$  is  given in  the next proposition.
\begin{proposition}\label{selfDuE}
Let $r$ and $s$ be positive integers. An abelian code $	\mathcal{C}$ in   ${\rm GR}(p^r,s)[ G]$ is  Euclidean   self-dual if and only if, in the decomposition (\ref{E3}), 
\begin{enumerate}[$i)$]
	\item $U_i$ is  Euclidean  self-dual    for    all $i=1,2,\dots,t_{\I}$,
    \item $V_j$ is  Hermitian  self-dual     for    all $j=1,2,\dots,t_{\II}$, and       
    \item $W_k^\prime = W_k^{\perp_{\rm E}}$    for all $k=1,2,\dots,t_{\III}$.
\end{enumerate}
\end{proposition}

The    characterization of Hermitian   self-dual abelian codes in ${\rm GR}(p^r,s)[G]$   follows immediately from (\ref{H3}) and (\ref{dualH}).
\begin{proposition}\label{selfDuH}
Let  $r$ be a positive integer and let $s$ be an  even positive  integer. Then an abelian code $	\mathcal{C}$ in  ${\rm GR}(p^r,s)[ G]$ is  Hermitian self-dual if and only if, in the decomposition~(\ref{H3}), 
\begin{enumerate}[$i)$] 
    \item $E_j$ is  Hermitian   self-dual    for   all $j=1,2,\dots,t_{\IIp}$, and       
    \item   $F_k^\prime= F_k^{\perp_{\rm E}}$ for all $k=1,2,\dots,t_{\IIIp}$.
\end{enumerate}
\end{proposition}

Necessary and sufficient conditions for the existence of Euclidean and Hermitian  self-dual abelian codes in  ${\rm GR}(p^r,s)[ G]$ are given as follows. The conditions for the Euclidean case have been  proven in \cite[Theorem 1.1]{W2002}. Here, we provide  an alternative constructive  proof.

\begin{proposition}\label{selfDEH}  Let   $r$ and $s$ be positive integers. Let  $G$ be a finite abelian group. Then
	there exists a Euclidean self-dual  abelian code  in ${\rm GR}(p^r,s)[G]$   if and only if one of the following statements holds,
	\begin{enumerate}[$i)$]
		\item $r$ is even, or
		\item $p=2$ and $|G|$ is even.
	\end{enumerate}

In addition, if $s$ is even, then  the conditions are equivalent to the existence of a Hermitian self-dual  abelian code  in ${\rm GR}(p^r,s)[G]$.
\end{proposition}
\begin {proof}
Assume that $G$ is decomposed  as $G=A\oplus P$, where $p\nmid |A|$ and $P$ is the Sylow $p$-subgroup of $G$ of order $p^a$, where $a\geq 0$.

From (\ref{E3}), assume that the code 
 \begin{align*}
	\mathcal{C}\cong \left(\prod_{i=1}^{t_{\I}}U_i \right)\times \left(\prod_{j=1}^{t_{\II}}V_j \right) \times \left(\prod_{k=1}^{t_{\III}}(W_k\times W_k^\prime) \right)
\end{align*}
is Euclidean self-dual in ${\rm GR}(p^r,s)[G]$.  Then, by Proposition~\ref{selfDuE}, $U_1$  is Euclidean self-dual in ${\rm GR}(p^r,s)[P]$. It follows that $|U_1|=(p^s)^{{rp^a}/{2}}$ and  ${rp^a}/{2}$ is an integer. Hence, $r$ is even, or $p=2$ and $a\geq 1$.

For the converse, if $r$ is even, then $p^{{r}/{2}}{\rm GR}(p^r,s)[G]$ is Euclidean self-dual. Assume that $r$ is odd, $p=2$ and $|G|$ is even. Let $r^\prime=\lceil \frac{r}{2}\rceil$.  Then   $|P|=2^a$ with $a\geq 1$ and $r=2r^\prime-1$. Since the order of $P$ is even, $P$ contains an element $x$ of order $2$. Define 
\begin{align*} 	
	\mathcal{C}\cong &\left(\prod_{i=1}^{t_{\I}}  \left(2^{r^\prime}{\rm GR}(2^r,s)[P]+2^{r^\prime-1}{\rm GR}(2^r,s)[P](Y^x+1)\right)   \right)\\
&\times  \left(\prod_{j=1}^{r_{\II}}   \left(2^{r^\prime}{\rm GR}(2^r,2se_j)[P]+2^{r^\prime-1}{\rm GR}(2^r,2se_j)[P](Y^x+1)\right)   \right)\\
	&\times\left(\prod_{k=1}^{t_{\III}} \left( {\rm GR}(2^r,sf_k)  [P]\times \{0\}\right) \right).
\end{align*}
We prove that $\mathcal{C}$ is Euclidean self-dual.  By Proposition~\ref{selfDuE}, it is sufficient to show that  \[U:=2^{r^\prime}{\rm GR}(2^r,s)[P]+2^{r^\prime-1}{\rm GR}(2^r,s)[P](Y^x+1) \]
 is Euclidean self-dual and 
\[V_j:=2^{r^\prime}{\rm GR}(2^r,2se_j)[P]+2^{r^\prime-1}{\rm GR}(2^r,2se_j)[P](Y^x+1)\] is Hermitian self-dual for all $j=1,2\dots,t_{\II}$.

  Let $\boldsymbol{u}=2^{r^\prime} \boldsymbol{e}  +2^{r^\prime-1} \boldsymbol{e}^\prime (Y^x+1) $ and $\boldsymbol{v}=  2^{r^\prime} \boldsymbol{f}  +2^{r^\prime-1} \boldsymbol{f} ^\prime(Y^x+1) $ be elements in $U$, where $\boldsymbol{e}$, $\boldsymbol{e}^\prime,$ $\boldsymbol{f},$ and $\boldsymbol{f}^\prime$ are in ${\rm GR}(2^r,s)[P]$.
Since $r=2r^\prime-1$ and $x=-x$, we have 
\begin{align*}
\langle \boldsymbol{u}, \boldsymbol{v}\rangle_{\rm E}
&= \langle 2^{r^\prime} \boldsymbol{e},  2^{r^\prime} \boldsymbol{f} \rangle_{\rm E}+
\langle 2^{r^\prime} \boldsymbol{e} ,  2^{r^\prime-1} \boldsymbol{f} ^\prime(Y^x+1) \rangle_{\rm E}\\
&~~~+
\langle 2^{r^\prime-1} \boldsymbol{e}^\prime (Y^x+1) , 2^{r^\prime} \boldsymbol{f} \rangle_{\rm E}+
\langle 2^{r^\prime-1} \boldsymbol{e}^\prime (Y^x+1) , 2^{r^\prime-1} \boldsymbol{f} ^\prime(Y^x+1)   \rangle_{\rm E}\\
&= 2^{r-1}\langle  \boldsymbol{e}^\prime (Y^x+1) ,   \boldsymbol{f} ^\prime(Y^x+1)   \rangle_{\rm E}\\
&=2^{r-1}\left( \langle  \boldsymbol{e}^\prime Y^x,   \boldsymbol{f} ^\prime Y^x \rangle_{\rm E}  + \langle  \boldsymbol{e}^\prime Y^x ,   \boldsymbol{f} ^\prime  \rangle_{\rm E}  
+ \langle  \boldsymbol{e}^\prime ,   \boldsymbol{f} ^\prime Y^x   \rangle_{\rm E} + \langle  \boldsymbol{e}^\prime   ,   \boldsymbol{f} ^\prime     \rangle_{\rm E}\right)\\
&=2^{r-1}\left( 2 \langle  \boldsymbol{e}^\prime ,   \boldsymbol{f} ^\prime   \rangle_{\rm E}  + 2\langle  \boldsymbol{e}^\prime Y^x ,   \boldsymbol{f} ^\prime  \rangle_{\rm E}   \right)\\
&=0.
\end{align*}
It is not difficult to verify that 
\begin{align*}
|U|
&=\frac{|2^{r^\prime}{\rm GR}(2^r,s)[P]||2^{r^\prime-1}{\rm GR}(2^r,s)[P](Y^x+1) |}{| (2^{r^\prime}{\rm GR}(2^r,s)[P])\cap (2^{r^\prime-1}{\rm GR}(2^r,s)[P](Y^x+1) )|} =\frac{(2^s)^{{(r^\prime-1)}2^a}   (2^s)^{{{r^\prime}2^a}/{2}}}  
  { (2^s)^{{{(r^\prime-1)}2^a}/{2}}} =(2^s)^{r2^{a-1}}.
\end{align*}
Therefore, $U$ is Euclidean self-dual.

Using arguments similar to the above, we can see that $V_j$ is Hermitian self-dual for all $j=1,2\dots,t_{\II}$.

For the Hermitian case, we assume that $s$ is even. 
The proof of the sufficiency is similar to the Euclidean case, except that   (\ref{H3}) and   Proposition~\ref{selfDuH} are applied instead of  (\ref{E3}) and Proposition~\ref{selfDuE} . 
 
For the converse, if $r$ is even, then $p^{{r}/{2}}{\rm GR}(p^r,s)[G]$ is Hermitian self-dual. Assume that $r$ is odd, $p=2$ and $|G|$ is even.   Then   $P$ contains an element $x$ of order $2$. 
Let $r^\prime=\lceil \frac{r}{2}\rceil$ and define
\begin{align*} 	
	\mathcal{C}\cong &
	\left(\prod_{j=1}^{r_{\IIp}}   \left(2^{r^\prime}{\rm GR}(2^r,s\acute{e}_j)[P]+2^{r^\prime-1}{\rm GR}(2^r,s\acute{e}_j)[P](Y^x+1)\right) \right)\times\left(\prod_{k=1}^{t_{\IIIp}} \left( {\rm GR}(2^r,s\acute{f}_k)  [P]\times \{0\}\right) \right).
\end{align*} 
By arguments similar to those in the proof of the Euclidean case,  we can verity  that  $2^{r^\prime}{\rm GR}(2^r,s\acute{e}_j)[P]+2^{r^\prime-1}{\rm GR}(2^r,s\acute{e}_j)[P](Y^x+1)$ is Hermitian self-dual for all $j=1,2,\dots,t_{\IIp} $. Therefore, $\mathcal{C}$ is Hermitian self-dual by Proposition~\ref{selfDuH}.
\end{proof}

\subsection{Enumeration of Self-Dual Abelian Codes} \label{subsec4.2}
We aim to characterize and enumerate Euclidean and Hermitian self-dual cyclic and abelian codes over Galois rings. For convenience, we fix the following notations. 
\begin{center}
    \begin{tabular}{lp{0.75\textwidth}}
        $\bullet ~NC({\rm GR}(p^r,s),n)$	&-- the number of cyclic codes of length $n$ over ${\rm GR}(p^r,s)$,\\
        $\bullet~NEC({\rm GR}(p^r,s),n)$  &-- the number of Euclidean self-dual cyclic codes of length $n$ over ${\rm GR}(p^r,s)$,\\
        $\bullet~NHC({\rm GR}(p^r,s),n)$  &-- the number of Hermitian self-dual cyclic codes of length $n$ over ${\rm GR}(p^r,s)$,\\
        $\bullet~NA({\rm GR}(p^r,s)[G])$  &-- the number of  abelian codes in ${\rm GR}(p^r,s)[G]$,\\
        $\bullet~NEA({\rm GR}(p^r,s)[G])$ &-- the number of Euclidean self-dual abelian codes in ${\rm GR}(p^r,s)[G]$,\\
        $\bullet~NHA({\rm GR}(p^r,s)[G])$ &-- the number of Hermitian self-dual abelian codes in ${\rm GR}(p^r,s)[G]$.\\
    \end{tabular}
\end{center}

To determine the numbers of Euclidean and Hermitian self-dual abelian codes, we need some group-theoretic  and number-theoretic  results. For completeness, we   recall the following  results. 
 
For a finite group $A$ and a positive integer $d$, let $\mathcal{N}_A(d)$  denote the number of elements  in $A$ of order $d$.   The explicit expression of  $\mathcal{N}_A(d)$  is completely determined in  \cite{B1997}.
 

Let $q$ be a prime power   and let $j$ be a positive integer. The pair $(j,q)$  is said to be {\em oddly good} if $j$ divides $q^t+1$ for some odd  integer $t\geq 1$, and {\em evenly good} if $j$ divides $q^t+1$ for some even  integer $t\geq 2$. It is said to be {\em good} if it is  oddly good or evenly good, and {\em bad} otherwise. The characterization of good and oddly-good pairs of integers can be found in \cite{J2018}, \cite{JLX2011}, \cite{JLLX2012},  \cite{JLS2012}, and \cite{M1997}. 

Let  $\chi$ and $\lambda$ be functions defined on the pair $(j,q)$, where $j$ is a positive integer, as follows.
\begin{align}\label{chi}
	\chi(j,q)
=\begin{cases}
0 &\text{ if } (j,q) \text{ is good},\\
1 &\text{ otherwise,}
\end{cases}
\end{align}
and
\begin{align}\label{lambda}
	\lambda(j,q)=
	         \begin{cases}
	            0&\text{if }  (j,q) \text{ is oddly good},\\
	            1&\text{otherwise}.
             \end{cases}
\end{align}
The following two lemmas are extended from the case where $q$ is a power of $2$ in \cite{JLLX2012} and \cite{JLS2012}  and the proof  is omitted. The readers may refer to \cite[Lemma 4.5]{JLLX2012} and \cite[Lemma 7]{JLS2012} for the idea of the  proofs. 
\begin{lemma}\label{good1} Let $A$ be a finite abelian group such that $\gcd(|A|,p)=1$ and let $h\in A$. Then $S_{p^s}(h)$ is of type $\III$ if and only if $({\rm ord}(h), p^s)$ is bad.
\end{lemma}
\begin{lemma}\label{good2} Let $A$ be a finite abelian group such that $\gcd(|A|,p)=1$  and let $h\in A\setminus\{0\}$. Then $S_{p^s}(h)$ is of type $\IIIp$ if and only if $({\rm ord}(h), p^{ {s}/{2}})$ is evenly good or bad.
\end{lemma}

Utilizing the decomposition in Section~\ref{sec2} and the discussion above, we  obtain the following formulas for the numbers of Euclidean and Hermitian self-dual abelian codes in ${\rm GR}(p^r,s)[G]$, where $G$ is an arbitrary finite abelian group.  Without loss of generality, we assume that $G=A\oplus P$, where $P$ is a finite abelian $p$-group and $A$ is a finite abelian group such that $p\nmid |A|$.
\begin{theorem} \label{NEHA}
Let $p$ be a prime and let $s,r$ be integers such that $1\leq s$ and $1\leq r$. Let $A$ be a finite abelian group of exponent $M$ such that $p\nmid M$ and let $P$ be a finite abelian $p$-group.
Then 
\begin{align*}NEA({\rm GR}(p^r,s)[A\oplus P])
&=
\left(NEA({\rm GR}(p^r,s)[P])\right)^{ \sum\limits_{{d\mid M, {\rm ord}_d(p^s)= 1}} (1-\chi(d,p^s))\mathcal{N}_A(d)} \\
&~~~~\times\prod_{\substack{d\mid M\\ {\rm ord}_d(p^s)\ne 1}} \left(NHA({\rm GR}(p^r,s\cdot{\rm ord}_d(p^s))[P])\right)^{(1-\chi(d,p^s))\frac{\mathcal{N}_A(d)}{{\rm ord}_d(p^s)}}  \\
&~~~~\times\prod_{d\mid M} \left(NA({\rm GR}(p^r,s\cdot{\rm ord}_d(p^s))[P])\right)^{\chi(d,p^s)\frac{\mathcal{N}_A(d)}{2{\rm ord}_d(p^s)}} .
\end{align*}
In addition, if $s$ is even, then
\begin{align*}
	NHA({\rm GR}(p^r,s)[A\oplus P])=
&\prod_{{d\mid M}} \left(NHA({\rm GR}(p^r,s\cdot{\rm ord}_d(p^s))[P])\right)^{(1-\lambda(d,p^\frac{s}{2}))\frac{\mathcal{N}_A(d)}{{\rm ord}_d(p^s)}}  \\
&\times\prod_{d\mid M} \left(NA({\rm GR}(p^r,s\cdot{\rm ord}_d(p^s))[P])\right)^{\lambda(d,p^\frac{s}{2})\frac{\mathcal{N}_A(d)}{2{\rm ord}_d(p^s)}}.
\end{align*}
\end{theorem}
\begin{proof}
First, we consider the Euclidean case.
From (\ref{E3}) and Proposition~\ref{selfDuE}, it is sufficient to count the numbers of Euclidean self-dual abelian codes $U_i$'s, the numbers of Hermitian self-dual abelian codes $V_i$'s, and the numbers of  abelian codes $W_i$'s  which     correspond  to  the $p^s$-cyclotomic classes of types $\I$, $\II$, and $\III$, respectively.

From \cite[Remark 2.5]{JLS2013},  we note that the elements in $A$ of the same order are partitioned into $p^s$-cyclotomic classes of the same type.
For each divisor $d$ of $M$, a  $p^s$-cyclotomic class containing an element of order $d$ has cardinality 
${{\rm ord}_d(p^s)}$, and hence, the number of such  $p^s$-cyclotomic classes   is $\frac{\mathcal{N}_A(d)}{{\rm ord}_d(p^s)}$.

For each divisor $d$ of $M$, we consider the following $3$ cases.

\noindent {\bf  Case 1.}   $\chi(d,p^s)=0$  and  ${\rm ord}_d(p^s)=1$. By Lemma~\ref{good1}, every  $p^s$-cyclotomic class    of $A$  containing an element  of order $d$  is of type  $\I$.  Since there are  $\frac{\mathcal{N}_A(d)}{{\rm ord}_d(p^s)}$ such $p^s$-cyclotomic classes,  the  number  of Euclidean self-dual abelian codes $U_i$'s corresponding to $d$ is 
\[\left(NEA({\rm GR}(p^r,s\cdot{\rm ord}_d(p^s))[P])\right)^{ \frac{\mathcal{N}_A(d)}{{\rm ord}_d(p^s)}}=\left(NEA({\rm GR}(p^r,s)[P])\right)^{   (1-\chi(d,p^s))\mathcal{N}_A(d)}  .\]

\noindent {\bf  Case 2.}   $\chi(d,p^s)=0$ and  ${\rm ord}_d(p^s)\ne 1$. By Lemma~\ref{good1}, every  $p^s$-cyclotomic class    of $A$  containing an element   of order $d$  is of type  $\II$.  Since there are  $\frac{\mathcal{N}_A(d)}{{\rm ord}_d(p^s)}$ such $p^s$-cyclotomic classes,  the  number  of Hermitian self-dual abelian codes $V_i$'s corresponding to $d$ is 
\begin{align*} \left(NHA({\rm GR}(p^r,s\cdot{\rm ord}_d(p^s))[P])\right)^{ \frac{\mathcal{N}_A(d)}{{\rm ord}_d(p^s)}} =\left(NHA({\rm GR}(p^r,s\cdot{\rm ord}_d(p^s))[P])\right)^{(1-\chi(d,p^s))\frac{\mathcal{N}_A(d)}{{\rm ord}_d(p^s)}}.
\end{align*}
 
\noindent {\bf  Case 3.}   $\chi(d,p^s)=1$. By Lemma~\ref{good1}, every  $p^s$-cyclotomic class    of $A$  containing an element   of order $d$  is of type  $\III$.  Since there are  $\frac{\mathcal{N}_A(d)}{{\rm ord}_d(p^s)}$ such $p^s$-cyclotomic classes,  the  number  of   abelian codes $W_i$'s corresponding to $d$ is 
 \[\left(NA({\rm GR}(p^r,s\cdot{\rm ord}_d(p^s))[P])\right)^{ \frac{\mathcal{N}_A(d)}{{\rm ord}_d(p^s)}}\left(NA({\rm GR}(p^r,s\cdot{\rm ord}_d(p^s))[P])\right)^{\chi(d,p^s)\frac{\mathcal{N}_A(d)}{2{\rm ord}_d(p^s)}}.\]
 Since $d$ runs over all divisors of $M$, we conclude the desired result.

For the Hermitian case, by Proposition~\ref{selfDuH}, it suffices to count the numbers  of Hermitian self-dual abelian codes $E_i$'s  and the numbers of  abelian codes $F_i$'s in (\ref{H3}) which    correspond  to  the $p^s$-cyclotomic classes of types $\IIp$ and $\IIIp$, respectively.    
 Considering the cases where $\lambda(d,p^\frac{s}{2})=1$   and where $ \lambda(d,p^\frac{s}{2})=0$,   the desired result can be obtained similarly to the Euclidean case, where Lemma~\ref{good2} is applied instead of Lemma~\ref{good1}.
\end{proof}

Note that, if $A$ is a cyclic group, the exponent $M$   is just the cardinality of $A$ and  $\mathcal{N}_A(d)$ is just $\phi(d)$, where $\phi$ is an Euler's totient function.

In Theorem~\ref{NEHA}, if $P$ is cyclic of order $p^a$, then the values $NA,NEA$ and $NHA$ may be replaced by $NC,NEC,$ and $NHC$, respectively. In general,   these values are not known in the literature.  Some special cases  where  $i)$ $\gcd(p,|G|)=1$;  and  $ii)$ $r=2$ and the Sylow $p$-subgroup of $G$ is cyclic   are  discussed in the following subsections.



\subsection{Self-Dual Abelian Codes in ${\rm GR}(p^r,s)[A]$, $\gcd(p,|A|)=1$}   \label{subsec4.3}
In this subsection, we complete the enumeration of Euclidean and Hermitian self-dual abelian codes in  ${\rm GR}(p^r,s)[A]$, where  $\gcd(p,|A|)=1$, or equivalently,     ${\rm GR}(p^r,s)[A]$ is a principal ideal group ring (see Proposition \ref{PIGR}).
If $A$ is cyclic,   this case is identical with  that of  simple root cyclic codes.

\begin{proposition}\label{PIGR}
	Let $p$ be a prime number and let $r,s$ be positive integers. Let $G$ be a finite abelian group. Then one of the following statements holds.
	\begin{enumerate}[$i)$]
		\item If $r=1$, then ${\rm GR}(p^r,s)[G]\cong \mathbb{F}_{p^s}[G]$ is a principal ideal ring if and only if the Sylow $p$-subgroup of $G$ is cyclic.
		\item If $r\geq 2$, then ${\rm GR}(p^r,s)[G]$ is a principal ideal ring if and only if $\gcd(p,|G|)=1$.
	\end{enumerate}
\end{proposition}

For $r=1$, the statement	has been proven in \cite{FiSe1976}.  For  $r\geq 2$, using notion of   morphic  rings (see the definition in \cite{CLZ2006}),  it has been shown that $\mathbb{Z}_ {p^r}[G]$ is principal ideal if and only if $\gcd(p,|G|)=1$ (see \cite[Theorem 1.2]{D2007} and \cite[Theorem 3.12 and Corollary 3.13]{CLZ2006}). The statements can be extended naturally to the case of ${\rm GR}(p^r,s)[G]$.

The enumerations of Euclidean and Hermitian self-dual abelian codes in a principal ideal group ring  ${\rm GR}(p^r,s)[A]$ is given as follows.
\begin{theorem} \label{NEHAP}
Let $p$ be a prime and let $s,r$ be positive integers. Let $A$ be a finite abelian group of exponent $M$ such that  $\gcd(p,|A|)=1$.
Then 
\begin{align*}NEA({\rm GR}(p^r,s)[A])=
	\begin{cases} (1+r)^{\sum\limits_{d\mid M}\chi(d,p^s)\frac{\mathcal{N}_A(d)}{2{\rm ord}_d(p^s)}}&\text{if }r \text{ is even,} \\
		0&\text{if }r \text{ is odd}.
	\end{cases}
\end{align*}
In addition, if $s$ is even, then
\begin{align*}
	NHA({\rm GR}(p^r,s)[A])=
 \begin{cases}(1+r)^{\sum\limits_{d\mid M}\lambda(d,p^{{s}/{2}})\frac{\mathcal{N}_A(d)}{2{\rm ord}_d(p^s)}}&\text{if }r \text{ is even,}\\
		0&\text{if }r \text{ is odd}.
	\end{cases}
\end{align*}
\end{theorem}
\begin{proof}
In ${\rm GR}(p^r,s)$, every ideal  can be regarded  as an abelian code in ${\rm GR}(p^r,s)[G]$ with $G=\{0\}$, and we have the following facts.
\begin{enumerate}[$i)$]
	\item The number of abelian codes in  ${\rm GR}(p^r,s)$ is $r+1$.
	\item If $r$ is odd, then there are   neither  Euclidean  self-dual abelian codes nor  Hermitian self-dual abelian codes   in ${\rm GR}(p^r,s)$.
	\item If $r$ is even, then  $r^{ {r}/{2}}{\rm GR}(p^r,s)$ is the only  Euclidean  self-dual  abelian code and it is the only  Hermitian self-dual abelian  code if $s$ is even.
\end{enumerate}
The above results hold true for any Galois extension of ${\rm GR}(p^r,s)$.

By considering $P=\{0\}$ in Theorem~\ref{NEHA}, the result follows immediately. 
\end{proof}

Note that, if $A$ is  cyclic,   $M$ and  $\mathcal{N}_A(d)$ can be replaced by   the cardinality of $A$ and $\phi(d)$, respectively, where $\phi$ is the Euler's totient function.

If $A$ is cyclic of order $n$ with $\gcd(n,p)=1$, then the number of Euclidean self-dual cyclic codes of length $n$ over ${\rm GR}(p^r,s)$ obtained in Theorem~\ref{NEHAP} is a special case of \cite[Theorem 5.7]{BGG2012} by viewing ${\rm GR}(p^r,s)$ as a finite chain ring of depth $r$.


\subsection{Self-Dual Abelian   Codes in ${\rm GR}(p^2,s)[A\oplus C_{p^a}]$}  \label{subsec4.4}
In this section, we restrict our study to the case where $r=2$ and $P=C_{p^a}$, a cyclic group of order $p^a$.  The enumerations of  Euclidean and Hermitian self-dual abelian codes in   ${\rm GR}(p^2,s)[A\oplus C_{p^a}]$ can be  obtained  as an application  of   Theorem~\ref{NEHA} and    some known results on  cyclic codes of length $p^a$ over   ${\rm GR}(p^2,s)$.

We recall some results on cyclic codes of length $p^a$ over ${\rm GR}(p^2,s)$.
The next lemma follows immediately from \cite[Corollary 3.9]{KLL2008} and  \cite[Theorem 3.6]{KLL2008}.
\begin{lemma}\label{N-cyclic2} 
The number of   cyclic codes of length $p^a$ over ${\rm GR}(p^2,s)$ is 
	\begin{align}
		NC({\rm GR}(p^2,s),p^a)=2\sum_{d=0}^{p^a-1} \frac{p^{s( \min\{\lfloor \frac{d}{2}\rfloor, p^{a-1}\}  +1)}-1}{p^s-1} +\frac{p^{s(  p^{a-1}  +1)}-1}{p^s-1}.
	\end{align}
\end{lemma}

\begin{proposition}[{\cite[Corollary 3.5]{KLL2012}}]\label{NEC}
    The number of Euclidean self-dual cyclic codes of length $2^a$ over ${\rm GR}(2^2,s)$ is
\[NEC({\rm GR}(2^2,s),2^a)=
\begin{cases}
	1 & \text{ if }a=1,\\
	1+2^s & \text{ if }a=2,\\
	1+2^s+2^{2s+1}\left( \frac{{(2^s)}^{(2^{a-2}-1)}-1}{2^s-1}\right)& \text{ if }a\geq 3.
\end{cases}\]
	If $p$ is an odd prime, then the number of Euclidean self-dual cyclic codes of length $p^a$ over ${\rm GR}(p^2,s)$ is
	\[NEC({\rm GR}(p^2,s),p^a)= 2\left( \frac{{(p^s)}^ {(p^{a-1}+1)}/{2}-1}{p^s-1}\right).\]
\end{proposition}

\begin{proposition}[{\cite[Theorem 3.5]{JLS2013}}] \label{NHC}
Let $p$ be a prime and let $s,a$ be positive integers such that $s$ is even. Then
the number of Hermitian  self-dual cyclic codes of length $p^a$ over ${\rm GR}(p^2,s)$ is
	\[NHC({\rm GR}(p^2,s),p^a)= \sum_{i_1=0}^{p^{a-1}} p^{{si_1}/{2}}= \frac{p^{ {s(p^{a-1}+1)}/{2}}-1}{p^{ {s}/{2}}-1}.\]
\end{proposition}

\bigskip
\begin{remark}\label{remCyclicSylow}
For cyclic codes of length $p^a$ over ${\rm GR}(p^2,s)$, the numbers $NC$, $NEC$, and $NHC$ have  already been determined in Lemma~\ref{N-cyclic2}, Proposition~\ref{NEC}, and Proposition~\ref{NHC}, respectively.  Combining these  results and  Theorem~\ref{NEHA},  the numbers $NEA({\rm GR}(p^2,s)[A\oplus C_{p^a}])$ and $NHA({\rm GR}(p^2,s)[A\oplus C_{p^a}])$ are explicitly determined.
\end{remark}

The numbers of Euclidean and Hermitian self-dual cyclic codes of arbitrary length $n$ over ${\rm GR}(p^2,s)$ can be obtained as a corollary of Remark~\ref{remCyclicSylow}. Some parts of the formulas can be simplified as in the next corollary.
\begin{corollary}
	Let $p$ be a prime and let $s,n$ be  positive integers. Write $n=mp^a$, where $a\geq 0$ and $p\nmid m$.
	Then 
	\begin{align*}NEC({\rm GR}(p^2,s),n)=&
	\left(NEC({\rm GR}(p^2,s),p^a)\right)^{\eta(m)}\\
     &\times \prod_{\substack{d\mid m\\ d\not\in\{ 1,2\}}} \left(NHC({\rm GR}(p^2,s\cdot{\rm ord}_d(p^s)),p^a)\right)^{(1-\chi(d,p^s))\frac{\phi(d)}{{\rm ord}_d(p^s)}}  \\
	&\times\prod_{d\mid m} \left(NC({\rm GR}(p^2,s\cdot{\rm ord}_d(p^s)),p^a)\right)^{\chi(d,p^s)\frac{\phi(d)}{2{\rm ord}_d(p^s)}},
	\end{align*}
	where 
	\[\eta(m)=
	\begin{cases}
	 1&\text{ if } m \text{ is odd},\\
	 2&\text{ if } m \text{ is even}.
	\end{cases}\]
	In addition, if $s$ is even, then
	\begin{align*}
		NHC({\rm GR}(p^2,s),n)=
	&\prod_{{d\mid m}} \left(NHC({\rm GR}(p^2,s\cdot{\rm ord}_d(p^s)),p^a)\right)^{(1-\lambda(d,p^\frac{s}{2}))\frac{\phi(d)}{{\rm ord}_d(p^s)}}  \\
	&\times\prod_{d\mid m} \left(NC({\rm GR}(p^2,s\cdot{\rm ord}_d(p^s)),p^a)\right)^{\lambda(d,p^\frac{s}{2})\frac{\phi(d)}{2{\rm ord}_d(p^s)}}.
	\end{align*}
\end{corollary}
\begin{proof}
Setting $r=2$ and A a cyclic group    of order $m$ in Theorem~\ref{NEHA}, the exponent of $A$ is $m$ and $\mathcal{N}_{A}(d)$ is just $\phi(d)$, where $\phi$ is the Euler's function.

Note that  $S_{p^s}(0)$ is the only $p^s$-cyclotomic class of $A$ of type $\I$ if $m$ is odd, and $S_{p^s}(0)$ and $S_{p^s}(\frac{m}{2})$ are the only $p^s$-cyclotomic classes of $A$ of type $\I$ if $m$ is even. Therefore, the values of  $\eta(m)$ follows.
\end{proof}

\section{Complementary Dual Abelian   Codes in ${\rm GR}(p^r,s)[G]$} \label{sec4p}

In this section, the  characterization and enumeration  of complementary dual abelian codes in the group ring  ${\rm GR}(p^r,s)[G]$ are given  based on the decomposition in Section \ref{sec2} and the theory of local group rings.

\subsection{Characterization and Enumeration of  Complementary Dual Abelian   Codes in ${\rm GR}(p^r,s)[P]$}

In this subsection,     we focus on   complementary dual abelian codes and direct summand ideals    in  each component  of ${\rm GR}(p^r,s)[P]$  in the decompositions \eqref{E1} and \eqref{H1}, where  $P$ is a finite abelian $p$-group.

First, we  recall some useful  definitions and properties  in ring theory.  For  a  finite commutative ring $R$ with identity, the {\em Jacobson radical}  of    $R$, denoted by  $Jac(R)$,   is defined to be the intersection of  all  maximal ideals of $R$.  The ring $R$ is said to be
\textit{local} if it has a unique maximal ideal.  


A local group ring has been  characterized in the following lemma.
\begin{lemma}[{\cite[Theorem]{N1972}}]  \label{lemLocal}Let $R$ be a commutative ring with identity and let $G$ be a finite  abelian group. Then  $R[G]$ is   local if and only if $R$ is local, $G$ is a $p$-group and
    $p\in Jac(R)$. 
\end{lemma}


\begin{proposition} \label{propLocal}
    	Let $p$ be a prime number and let $r,s$ be positive integers. Let $P$ be a finite abelian $p$-group.  Then 
    ${\rm GR}(p^r,s)[P]$  is a local group ring.
\end{proposition}
\begin{proof} Since the ideal $\langle p \rangle $ is the unique maximal ideal of  ${\rm GR}(p^r,s)$, the ring  ${\rm GR}(p^r,s)$ is local.   Moreover,  $p\in \langle p \rangle =Jac( {\rm GR}(p^r,s)$.  By Lemma \ref{lemLocal}, ${\rm GR}(p^r,s)[P]$ is a local group ring.    
     \end{proof}

By Proposition  \ref{propLocal}, ${\rm GR}(p^r,s)[P]$ is local.  Denote by  $M$  the maximal ideal of  ${\rm GR}(p^r,s)[P]$.  The characterizations of the Euclidean and Hermitian  complementary dual abelian codes and  the direct summands in a local group ring  ${\rm GR}(p^r,s)[P]$   are  given in the following theorems.

\begin{theorem} \label{complementary}
    Let $p$ be a prime number and let $r,s$ be positive integers. Let $P$ be a finite abelian $p$-group. 
      Then $\{0\}$ and ${\rm GR}(p^r,s)[P]$ are the only Euclidean complementary  dual abelian codes in  ${\rm GR}(p^r,s)[P]$.
\end{theorem}
\begin{proof}  Clearly, $\{0\}$ and ${\rm GR}(p^r,s)[P]$  are Euclidean  complementary  dual abelian codes in  ${\rm GR}(p^r,s)[P]$. 
    Let $\mathcal{C}$ be an  abelian code in  ${\rm GR}(p^r,s)[P]$  such that  $\{0\}\subsetneq \mathcal{C} \subsetneq {\rm GR}(p^r,s)[P]$. 
    Then $\mathcal{C}\subseteq M$. It follows that  $M^{\perp_{\rm E}} \subseteq \mathcal{C}^{\perp_{\rm E}}  \subseteq M $ which implies  $M^{\perp_{\rm E}}  \subseteq \mathcal{C} \subseteq M$. 
    Hence,  $\{0\}\ne M ^{\perp_{\rm E}}  \subseteq \mathcal{C}\cap \mathcal{C}^{\perp_{\rm E}} \subseteq M$. Consequently, $\mathcal{C}$ is not  Euclidean complementary dual. Therefore, the ideals $\{0\}$ and ${\rm GR}(p^r,s)[P]$ are the only Euclidean complementary  dual abelian codes in  ${\rm GR}(p^r,s)[P]$.
\end{proof}

It is not difficult to see that the proof of Theorem \ref{complementary} is independent of the inner product. Hence, we have the following corollary. 

\begin{corollary} \label{Hcomplementary}
Let $p$ be a prime number and let $r,s$ be positive integers such that $s$ is even. Let $P$ be a finite abelian $p$-group. 
Then $\{0\}$ and ${\rm GR}(p^r,s)[P]$ are the only Hermitian  complementary  dual abelian codes in  ${\rm GR}(p^r,s)[P]$.
\end{corollary}

\begin{theorem}\label{derectsummand}
 Let $p$ be a prime number and let $r,s$ be positive integers. Let $P$ be a finite abelian $p$-group. 
   Then  ideals $\{0\}$ and ${\rm GR}(p^r,s)[P]$ are the only direct summands in  ${\rm GR}(p^r,s)[P]$.
\end{theorem}
\begin{proof}
    Clearly, $\{0\}$ and ${\rm GR}(p^r,s)[P]$  are  direct summands in  ${\rm GR}(p^r,s)[P]$.
 Let $\{0\}\subsetneq \mathcal{C} \subsetneq {\rm GR}(p^r,s)[P]$  be  an  ideal in   ${\rm GR}(p^r,s)[P]$. Suppose that $\mathcal{C}$ is  a direct summand. Then there exist an ideal $\mathcal{C}^\prime $ in  ${\rm GR}(p^r,s)[P]$ such that $\mathcal{C}\cap \mathcal{C}^\prime=\{0\}$ and $\mathcal{C}+\mathcal{C}^\prime = {\rm GR}(p^r,s)[P]$. Since $M$ is the maximal ideal in  ${\rm GR}(p^r,s)[P]$, we have   $\mathcal{C}\subseteq  M$ and $\mathcal{C}^\prime \subseteq  M$.  Hence, $\mathcal{C}+\mathcal{C}^\prime\subseteq M \subsetneq  {\rm GR}(p^r,s)[P]$, a contradiction. Therefore,  the ideals $\{0\}$ and  ${\rm GR}(p^r,s)[P]$ are the only direct summands in   ${\rm GR}(p^r,s)[P]$.
\end{proof}

The  above results can be summarized as follows. 
\begin{corollary}
Let $p$ be a prime number and let $r,s$ be positive integers. Let $P$ be a finite abelian $p$-group.   Then the following statements hold.
    \begin{enumerate}
        \item    The number of Euclidean complementary dual abelian codes in ${\rm GR}(p^r,s)[P]$    is $2$.
        \item    If $s$ is even, the number of Hermitian  complementary dual abelian codes in ${\rm GR}(p^r,s)[P]$   is~$2$. 
          \item    The number of  direct summand ideals in ${\rm GR}(p^r,s)[P]$  is $2$. 
    \end{enumerate}
\end{corollary}

\subsection{Characterization  and Enumeration of Complementary Dual Abelian   Codes in ${\rm GR}(p^r,s)[G]$}

In this subsection,  we focus on the characterization and enumeration of complementary dual abelian codes in ${\rm GR}(p^r,s)[G]$, where $G$ is an arbitrary finite abelian group.  Using the decompositions in Section \ref{sec2} and results in the previous subsection,   the  characterization and enumeration  of such complementary dual codes are given independent of $r$ and the Sylow $p$-subgroup  of $G$.

    Recall that $G\cong A\times P$,  where $P$ is the Sylow $p$-subgroup of $G$ and $A$ is its complement subgroup.  The group ring   ${\rm GR}(p^r,s)[G]$ is viewed as  ${\rm GR}(p^r,s)[G]\cong \mathcal{R}[P]$, where $\mathcal{R}={\rm GR}(p^r,s)[A]$.  Using the decomposition of   ${\rm GR}(p^r,s)[G]$ in \eqref{E2},  the characterization of  a Euclidean complementary  dual abelian code in  ${\rm GR}(p^r,s)[G]$ can be concluded via   \eqref{E3} and \eqref{dualE}  as follows.
    
    \begin{proposition} \label{charLCDE} Let $p$ be a prime number and let $r,s$ be positive integers.   Let  $A$ be  finite abelian group such that $p\nmid |A|$ and let  $P$ be a finite abelian $p$-group. Then an abelian code $\mathcal{C}$ in ${\rm GR}(p^r,s)[A\times P]  $ decomposed as in  \eqref{E3} is Euclidean complementary dual if and only if the following statements hold.
        \begin{enumerate}
            \item  $U_i$ is Euclidean  complementary dual   for all $ i=1,2,\dots, r_\I$.
            \item  $V_j$ is Hermitian  complementary dual   for all $  j=1,2,\dots, r_\II$.
            \item  $W_k\cap (W_k^\prime)^{\perp_{\rm E}}=\{0\}$ and    $W_k^\prime \cap W_k^{\perp_{\rm E}}=\{0\}$  for all $k=1,2,\dots, r_{\III}$.
        \end{enumerate}
    \end{proposition}

    The next corollary follows immediately from  Theorem \ref{complementary}, Proposition \ref{charLCDE}, and  Corollary~\ref{Hcomplementary}.
    
    \begin{corollary} \label{corCharLCDE}
        Let $p$ be a prime number and let $r,s$ be positive integers.   Let  $A$ be  finite abelian group such that $p\nmid |A|$ and let  $P$ be a finite abelian $p$-group. Then an abelian code $C$ in ${\rm GR}(p^r,s)[A\times P]  $ decomposed as in  \eqref{E3} is Euclidean complementary dual if and only if the following statements hold.
        \begin{enumerate}
            \item  $U_i  \in \{ \{0\}, {\rm GR}(p^r,s)[P] \}$  for all $i=1,2,\dots, r_\I$.
            \item  $V_j\in \{\{0\}, {\rm GR}(p^r,2se_j)[P] \}$      for all $ j=1,2,\dots, r_\II$.
            \item  $(W_k, W_k^\prime ) \in \{ (\{0\}, \{0\}), ({\rm GR}(p^r,sf_k)[P], {\rm GR}(p^r,sf_k)[P])\}$  for all $k=1,2,\dots, r_{\III}$.
        \end{enumerate}
    \end{corollary}

    From Corollary \ref{corCharLCDE}, it is not difficult to see that  the number of Euclidean complementary dual abelian codes in ${\rm GR}(p^r,s)[A\times P]  $  is  independent of $r$ and  the group $P$ and it can be  determined in the following corollary. 
    \begin{corollary}\label{corNumLCDE}
         Let $p$ be a prime number and let $r,s$ be positive integers.   Let  $A$ be  finite abelian group such that $p\nmid |A|$ and let  $P$ be a finite abelian $p$-group.   If the exponent of $A$ is $M$ and  ${\rm GR}(p^r,s)[A\times P]  $ is decomposed as in  \eqref{E2}, then the number of Euclidean complementary dual abelian codes in ${\rm GR}(p^r,s)[A\times P]  $ is  \[2^{r_{\I}+r_{\II} +r_\III} = 2^{  \sum\limits_{{d\mid M}}  {(1-\chi(d,p^s))\frac{\mathcal{N}_A(d)}{{\rm ord}_d(p^s)}}    + \sum\limits _{d\mid M}  {\chi(d,p^s)\frac{\mathcal{N}_A(d)}{2{\rm ord}_d(p^s)}}} ,\]
         where $\mathcal{N}_A(d)$  denote the number of elements  in $A$ of order $d$.   
    \end{corollary}
\begin{proof}
    The first part follows from Corollary \ref{corCharLCDE}. The equality can be derived similar to the proof of 
 Theorem \ref{NEHA}.
\end{proof}

       Using the decomposition of   ${\rm GR}(p^r,s)[G]$ in \eqref{H2},  the characterization of  a  Hermitian  complementary  dual abelian code in  ${\rm GR}(p^r,s)[G]$ can be concluded via   \eqref{H3} and \eqref{dualH}  in the following proposition. 
    
    \begin{proposition} \label{HcharLCDE}  Let $p$ be a prime number and let $r,s$ be positive integers such that $s$ is even.   Let  $A$ be  finite abelian group such that $p\nmid |A|$ and let  $P$ be a finite abelian $p$-group. Then an abelian code $\mathcal{C}$ in ${\rm GR}(p^r,s)[A\times P]  $ decomposed as in   \eqref{H3} is Hermitian complementary dual if and only if the following statements hold.
        \begin{enumerate}
            \item  $E_j$ is Hermitian  complementary dual   for all $ j=1,2,\dots, r_\IIp$.
            \item  $F_k\cap (F_k^\prime)^{\perp_{\rm E}}=\{0\}$ and    $F_k^\prime \cap F_k^{\perp_{\rm E}}=\{0\}$  for all $ j=1,2,\dots, r_{\IIIp}$.
        \end{enumerate}
    \end{proposition}

    The following result follows directly from Proposition \ref{HcharLCDE} and  Corollary \ref{Hcomplementary}. 
    
    \begin{corollary} \label{HcorCharLCDE}
        Let $p$ be a prime number and let $r,s$ be positive integers such that $s$ is even.   Let  $A$ be  finite abelian group such that $p\nmid |A|$ and let  $P$ be a finite abelian $p$-group. Then an abelian code $\mathcal{C}$ in ${\rm GR}(p^r,s)[A\times P]  $ decomposed as in   \eqref{H3} is Hermitian complementary dual if  and only if the following statements hold.
        \begin{enumerate}
            \item  $E_j\in \{\{0\},  {\rm GR}(p^r,s\acute{e}_j)[P]\}$      for all $j=1,2,\dots, r_\IIp$.
            \item  $(F_k, F_k^\prime ) \in \{ (\{0\},  \{0\}), ( {\rm GR}(p^r,s\acute{f}_k)[P],{\rm GR}(p^r,s\acute{f}_k)[P])\}$  for all $ k=1,2,\dots, r_{\IIIp}$.
        \end{enumerate}
    \end{corollary}

    From Corollary \ref{HcorCharLCDE}, the number of Hermitian complementary dual abelian codes in ${\rm GR}(p^r,s)[A\times P]  $ is  independent of  $r$ and the group  $P$ and it is given  in the following corollary. 
    \begin{corollary}\label{HcorNumLCDE}
        Let $p$ be a prime number and let $r,s$ be positive integers such that $s$ is even.   Let  $A$ be  finite abelian group such that $p\nmid |A|$ and let  $P$ be a finite abelian $p$-group.   If the exponent of $A$ is $M$ and  ${\rm GR}(p^r,s)[A\times P]  $ is decomposed as in  \eqref{H2},   then the number of Hermitian complementary dual abelian codes in ${\rm GR}(p^r,s)[A\times P] $ is  \[2^{r_{\Ip}+r_{\IIp}}= 2^{  \sum\limits _{{d\mid M}}  {(1-\lambda(d,p^\frac{s}{2}))\frac{\mathcal{N}_A(d)}{{\rm ord}_d(p^s)}} + \sum\limits_{d\mid M}  {\lambda(d,p^\frac{s}{2})\frac{\mathcal{N}_A(d)}{2{\rm ord}_d(p^s)}}  },\]
 where $\mathcal{N}_A(d)$  denote the number of elements  in $A$ of order $d$.   
\end{corollary}
\begin{proof}
The first part follows from Corollary \ref{HcorCharLCDE}. The equality can be derived similar to the proof of 
Theorem \ref{NEHA}.
\end{proof}


\section{Conclusion} \label{sec6}

Self-dual and complementary dual   abelian codes in  ${\rm GR}(p^r,s)[G]$ have been studied with respect to the  Euclidean and Hermitian inner products, a group ring of  a finite abelian group $G$ over a Galois ring ${\rm GR}(p^r,s)$.    We have characterized  such self-dual codes as well as determined  necessary and sufficient conditions for  ${\rm GR}(p^r,s)[G]$ to contain  a   Euclidean (resp, Hermitian) self-dual abelian code. For any finite abelian group $G$ and Galois ring ${\rm GR}(p^r,s)$,  the enumerations of such self-dual codes have been given.
    In the case where    $\gcd(|G|,p)=1$, the enumeration has been completed by restricting the Sylow $p$-subgroup to be $\{0\}$.  Applying some known results on cyclic codes of length $p^a$ over ${\rm GR}(p^2,s)$, we have determined  explicitly  the numbers of Euclidean and Hermitian  self-dual  abelian codes in ${\rm GR}(p^2,s)[G]$  if the Sylow $p$-subgroup of $G$ is cyclic.
As corollaries,  analogous results   on Euclidean and Hermitian  self-dual cyclic  codes    over ${\rm GR}(p^r,s)$  have been concluded. Subsequently, the characterization and  enumeration of complementary dual abelian codes   in ${\rm GR}(p^r,s)[G]$  have been  given. The number of  complementary dual abelian codes   in ${\rm GR}(p^r,s)[G]$ has been  shown to be independent of $r$ and the Sylow $p$-subgroup of $G$.

It would be interesting to study the unknown terms in Theorem \ref{NEHA} and extend the results to abelian codes over finite chain rings or the case where the Sylow $p$-subgroup of the group is not cyclic.


\section*{Appendix A}
\renewcommand{\thesection}{A}
In this appendix, we discuss   the Euclidean and Hermitian duals of abelian codes in  ${\rm GR}(p^r,s)[G]$.  First, we recall that $G\cong A\times P$, where $P$ is the Sylow $p$-subgroup of $G$ and $A$ is a complementary subgroup of $P$ in $G$.   The group ring $\mathcal{R}:={\rm GR}(p^r,s)[A]$ is decomposed as in (\ref{E1}) or (\ref{H1}), and ${\rm GR}(p^r,s)[G]\cong \mathcal{R}[P]$.
\subsection{Euclidean Duality}\label{SSAEC}

Let $\uppsi$ denote the isomorphism in (\ref{E1}). For each  element $\boldsymbol{x}\in \mathcal{R}$, we  can  write
	\begin{align}\label{xx1}
		\uppsi(\boldsymbol{x})=(x_1,\dots, x_{r_{\I}},  y_1,\dots, y_{r_{\II}},     z_1,z_1^\prime,\dots, z_{r_{\III}} , z_{r_{\III}}^\prime ),
	\end{align}	
	where $x_i\in  {\rm GR}(p^r,s)$, $y_j\in {\rm GR}(p^r,2se_j)$,  and $z_k,z_k^\prime\in {\rm GR}(p^r,sf_k)$ for all $i=1,2,\dots,r_{\I}$, $j=1,2,\dots,r_{\II}$, and  $k=1,2,\dots,r_{\III}$.
	 
We are going to view $\widehat{\boldsymbol{x}}$ defined in Section~\ref{sec2}  in terms of  (\ref{xx1}). We note that, for $\boldsymbol{c}=\sum_{a\in A} c_{a}Y^a \in {\rm GR}(p^2,s)[A]$,  we have 
\[\widehat{\boldsymbol{c}}=\sum_{a\in A} c_{a}Y^{-a}=  \sum_{a\in A} c_{-a}Y^a.\] 
Then $\Breve{\widehat{\boldsymbol{c}}}= \sum_{a\in A} \Breve{d}_{a}Y^{a},$ where
\[\Breve{d}_{a}=\sum_{h\in A}  {c}_{-h} \zeta^{\gamma_{a}(h)} . \]

From (\ref{ca}),  we can see that,  if $S_{p^s}(h)$ is of type $\I$, then
\begin{align}\label{mark1} \Breve{d}_{h}=\Breve{c}_{h}, \end{align}
 and  if $S_{p^s}(h)$ is of type $\II$ with cardinality $2\nu$,  then $-h=p^{s\nu}\cdot h$ by Remark \ref{rem1}. It follows that   
\begin{align}\label{mark2}
\Breve{d}_{h}=\sum_{a\in A} {c}_{-a} \zeta^{\gamma_{h}(a)}=\sum_{a\in A} {c}_{a} \zeta^{\gamma_{-h}(a)}=\sum_{a\in A} {c}_{a} \zeta^{\gamma_{p^{s\nu}\cdot h}(a)}=\sum_{a\in A} {c}_{a} \left(\zeta^{\gamma_h(a)}\right)^{p^{s\nu}} =\theta({\Breve{c}_{h}})  ,
\end{align}
where $\theta({\alpha})=\alpha_0^{p^{s\nu}}+\alpha_1^{p^{s\nu}}p+\dots+\alpha_{r-1}^{p^{s\nu}}p^{r-1}$
for all  $\alpha=\alpha_0+\alpha_1p+\dots+\alpha_{r-1}p^{r-1}$.

Therefore, by the isomorphism $\uppsi$ (see also \cite{KR2003}),   the following properties are obtained. 
\begin{enumerate}
	\item  From (\ref{mark1}), the involution $\widehat{~}$ induces the identity automorphism on ${\rm GR}(p^r,s)$.
	\item  From (\ref{mark2}), the involution $\widehat{~}$ induces the ring automorphism $\bar{~}$ on ${\rm GR}(p^r,2se_j)$  as defined in (\ref{aut}), $i.e.$,   	
		\begin{align*}
			\overline{\alpha}=\alpha_0^{p^{s{e}_j}}+\alpha_1^{p^{s{e}_j}}p+\dots+\alpha_{r-1}^{p^{s{e}_j}}p^{r-1}
		\end{align*}
			for all  $\alpha=\alpha_0+\alpha_1p+\dots+\alpha_{r-1}p^{r-1}$ in 		  ${\rm GR}(p^r,2s{e}_j)$,  
				where $\alpha_i\in \mathcal{T}_{2s{e}_j}$ for all $i=0,1,\dots,r-1$.  

	\item  For each pair $(z, z^\prime)\in {\rm GR}(p^r,sf_k)\times {\rm GR}(p^r,sf_k)$,    we have
	\begin{align*}
	\widehat{\uppsi^{-1}(z,z^\prime)}
	 =\uppsi^{-1}(z^\prime,z).		
	\end{align*}
\end{enumerate}

From the discussion,   we have
 	\begin{align*}
	\uppsi(\widehat{\boldsymbol{x}})=(x_1,\dots, x_{r_{\I}},\overline{y_1},\dots, \overline{y_{r_{\II}}}, z_1^\prime, {z_1},\dots, z_{r_{\III}}^\prime ,{z_{r_{\III}}} ),
	\end{align*}
where $\bar{~}$ is induced as above in an appropriate Galois extension.

\begin{proposition}\label{orthogonality1}
	Let  $\boldsymbol{x}=\sum_{b\in P}\boldsymbol{x}_b Y^b $ and $\boldsymbol{u}=\sum_{b\in P}\boldsymbol{u}_b Y^b$ be elements in $\mathcal{R}[P]$.	
Decomposing $\boldsymbol{x}_b, \boldsymbol{u}_b$  using (\ref{xx1}),  	we have
	\begin{align*}
		\uppsi(\boldsymbol{x}_b)=(x_{b,1},\dots, x_{b,r_{\I}},y_{b,1},\dots, y_{b,r_{\II}},  z_{b,1},z_{b,1}^\prime,\dots, z_{b,r_{\III}} , z_{b,r_{\III}}^\prime  )
	\end{align*}
	and
	\begin{align*}
		\uppsi(\boldsymbol{u}_b)=(u_{b,1},\dots, u_{b,r_{\I}},v_{b,1},\dots, v_{b,r_{\II}}, w_{b,1},w_{b,1}^\prime,\dots, w_{b,r_{\III}} , w_{b,r_{\III}}^\prime  ).
	\end{align*}
Then 
\begin{align*}
	\uppsi(\langle 	\boldsymbol{x},	\boldsymbol{u}\rangle\displaystyle\hspace*{-0.4ex}{~ \atop \widehat{~} })
    &=\uppsi\left(\sum_{b\in P}\boldsymbol{x}_b\widehat{\boldsymbol{u}_b}\right)=\sum_{b\in P}\uppsi(\boldsymbol{x}_b)\uppsi(\widehat{\boldsymbol{u}_b})\\
    &=\left(\sum_{b\in P}x_{b,1}{u_{b,1}} ,\dots,\sum_{b\in P}x_{b,r_{\I}}{u_{b,r_{\I}}},\sum_{b\in P}y_{b,1}\overline{v_{b,1}} , \dots,\sum_{b\in P}y_{b,r_{\II}}\overline{v_{b,r_{\II}}} , \right.\\
	&~~~~\left. \sum_{b\in P}z_{b,1}{w_{b,1}^\prime},\sum_{b\in P} z_{b,1}^\prime{w_{b,1}},  \dots,\sum_{b\in P}z_{b,r_{\III}}{w_{b,r_{\III}}^\prime}, \sum_{b\in P} z_{b,r_{\III}}^\prime{w_{b,r_{\III}}}\right).                                      
\end{align*}

In particular, $\langle 	\boldsymbol{x},	\boldsymbol{u}\rangle\displaystyle\hspace*{-0.4ex}{~ \atop \widehat{~} }=0$ if and only if $	\uppsi(\langle 	\boldsymbol{x},	\boldsymbol{u}\rangle\displaystyle\hspace*{-0.4ex}{~ \atop \widehat{~} })=\boldsymbol{0}$, or equivalently,
\[\sum_{b\in P} x_{b,j}{u_{b,j}}=0 \text{ for all } j=1,2,\dots, r_{\I} ,~~~ \sum_{b\in P} y_{b,j}\widetilde{v_{b,j}}=0 \text{ for all } j=1,2,\dots, r_{\II}, \]
and 
\[\sum_{b\in P} z_{b,j}{w_{b,j}^\prime}=0=\sum_{b\in P}z_{b,j}^\prime{w_{b,j}} \text{ for all } j=1,2,\dots, r_{\III}. \]
\end{proposition}

Using  the orthogonality in Proposition~\ref{orthogonality1}, the Euclidean dual of $\mathcal{C}$ in (\ref{E3}) can be viewed to be of the form
\begin{align} 
	\mathcal{C}^{\perp_{\rm E}}\cong \left(\prod_{i=1}^{t_{\I}} U_i^{\perp_{\rm E}}  \right)&\times \left(\prod_{j=1}^{t_{\II}} V_j ^{\perp_{\rm H}} \right)\times \left(\prod_{k=1}^{t_{\III}} \left( (W_k^\prime) ^{\perp_{\rm E}}\times  W_k^{\perp_{\rm E}}\right) \right).
\end{align}

\subsection{Hermitian Duality}\label{SSAHC}

Let $\psi$ denote the isomorphism in (\ref{H1}). Then each element $\boldsymbol{x}\in \mathcal{R}$, we  can   write
	\begin{align}\label{xx}
		\psi(\boldsymbol{x})=(x_1,\dots, x_{t_{\IIp}},  y_1,y_1^\prime,\dots, y_{t_{\IIIp}} , y_{t_{\IIIp}}^\prime ),
	\end{align}	
	where $x_j\in {\rm GR}(p^r,s\acute{e}_j)$ and  $y_k, y_k^\prime\in {\rm GR}(p^r,s\acute{f}_k)$ for all $j=1,2,\dots,t_{\IIp}$ and  $k=1,2,\dots,t_{\IIIp}$.

We note that, for $\boldsymbol{c}=\sum_{a\in A} c_{a}Y^a \in {\rm GR}(p^r,s)[A]$,  we have 
	\[\widetilde{\boldsymbol{c}}=\sum_{a\in A} \overline{c_{a}}Y^{-a}=  \sum_{a\in A} \overline{c_{-a}}Y^a,\] 
	where $\overline{\alpha_0+p\alpha_1+\dots+p^{r-1}\alpha_{r-1}}=\alpha_0^{p^{{s}/{2}}}+p\alpha_1^{p^ {{s}/{2}}} +\dots+ p^{r-1}\alpha_{r-1}^{p^ {{s}/{2}}}$.
	Then $\Breve{\widetilde{\boldsymbol{c}}}= \sum_{a\in A} \Breve{w}_{a}Y^{a},$ where
	$\Breve{w}_{a}=\sum_{h\in A}  \overline{{c}_{-h}} \zeta^{\gamma_{a}(h)} $. 
	 
 From (\ref{ca}),    if $S_{p^s}(h)$ is of type $\IIp$  with cardinality $\nu$,  then $-a=p^{{s\nu}/{2}}\cdot a$ by Remark \ref{rem1}. Since $\nu$ is odd, we have 
\begin{align}\label{mark3}	
\Breve{w}_{h} 
=\sum_{a\in A} \overline{{c}_{-a}} \zeta^{\gamma_{h}(a)}=\sum_{a\in A} \overline{{c}_{a}} \zeta^{\gamma_{-h}(a)}=\sum_{a\in A} \overline{{c}_{a}}\zeta^{\gamma_{p^{{{s\nu}/{2}}}\cdot h}(a)}=\sum_{a\in A} \overline{{c}_{a}}\left(\zeta^{\gamma_{h}(a)}\right)^{p^{{{s\nu}/{2}}}}=\theta({\Breve{c}_{h}}),
\end{align}
where $\theta({\alpha})=\alpha_0^{p^{{s\nu}/{2}}}+\alpha_1^{p^{{s\nu}/{2}}}p
+\dots+\alpha_{r-1}^{p^{{s\nu}/{2}}}p^{r-1}$ 
for all  $\alpha=\alpha_0+\alpha_1p+\dots+\alpha_{r-1}p^{r-1}$.%

By the isomorphism $\psi$ (see also \cite{KR2003}),  we have the following properties. 
\begin{enumerate}
\item By (\ref{mark3}), the involution $\widetilde{~}$ induces the ring automorphism $\bar{~}$ on ${\rm GR}(p^r,s\acute{e}_j)$  as defined in (\ref{aut}), $i.e.$,   	
	\begin{align*}
		\overline{\alpha}=\alpha_0^{p^{{s\acute{e}_j}/{2}}}+\alpha_1^{p^{{s\acute{e}_j}/{2}}}p+\dots+\alpha_{r-1}^{p^{{s\acute{e}_j}/{2}}}p^{r-1}
	\end{align*}
		for all  $\alpha=\alpha_0+\alpha_1p+\dots+\alpha_{r-1}p^{r-1}$ in 		  ${\rm GR}(p^r,s\acute{e}_j)$,  
			where $a_i\in \mathcal{T}_{s\acute{e}_j}$ for all $i=0,1,\dots,r-1$.  
		 
\item  For each pair $(z, z^\prime)\in {\rm GR}(p^r,s\acute{f}_k)\times {\rm GR}(p^r,s\acute{f}_k)$,  we have 
\begin{align*}
\widetilde{\psi^{-1}(z,z^\prime)}
 =\psi^{-1}(z^\prime,z).
\end{align*}
\end{enumerate}

Hence,     $\widetilde{\boldsymbol{x}}$ defined in Section~\ref{sec2}  can be viewed in terms of  (\ref{xx}) as 
 	\begin{align*}
 	\psi(\widetilde{\boldsymbol{x}})=(\overline{x_1},\dots, \overline{x_{t_{\IIp}}},  {y_1^\prime}, {y_1},\dots,{y_{r_{\II}}^\prime} ,{y_{t_{\IIIp}}} ).
	\end{align*}
where $\bar{~}$ is induced as above in an appropriate Galois extension.

\begin{proposition}\label{orthogonality}
	Let  $\boldsymbol{x}=\sum_{b\in P}\boldsymbol{x}_b Y^b $ and $\boldsymbol{u}=\sum_{b\in P}\boldsymbol{u}_b Y^b$ be elements in $\mathcal{R}[P]$.	
Decomposing $\boldsymbol{x}_b, \boldsymbol{u}_b$  using (\ref{xx}),  we have
\begin{align*}
	\psi(\boldsymbol{x}_b)=(x_{b,1},\dots, x_{b,t_{\IIp}}, y_{b,1},y_{b,1}^\prime,\dots, y_{b,t_{\IIIp}} , y_{b,t_{\IIIp}}^\prime  )
\end{align*}
and
\begin{align*}
	\psi(\boldsymbol{u}_b)=(u_{b,1},\dots, u_{b,t_{\IIp}},v_{b,1},v_{b,1}^\prime,\dots, v_{b,t_{\IIIp}} , v_{b,t_{\IIIp}}^\prime  ).
\end{align*} 
Then 
\begin{align*}
	\psi(\langle 	\boldsymbol{x},	\boldsymbol{u}\rangle_{\sim})
    &=\psi\left(\sum_{b\in P}\boldsymbol{x}_b\widetilde{\boldsymbol{u}_b}\right)=\sum_{b\in P}\psi(\boldsymbol{x}_b)\psi(\widetilde{\boldsymbol{u}_b})\\
	&=\left(\sum_{b\in P}x_{b,1}\overline{u_{b,1}} ,\dots,\sum_{b\in P}x_{b,t_{\IIp}}\overline{u_{b,t_{\IIp}}} , \sum_{b\in P}y_{b,1} {v_{b,1}^\prime},\sum_{b\in P} y_{b,1}^\prime {v_{b,1}},  \dots,\sum_{b\in P}y_{b,t_{\IIIp}} {v_{b,t_{\IIIp}}^\prime}, \sum_{b\in P} y_{b,t_{\IIIp}}^\prime {v_{b,t_{\IIIp}}}\right).                                         
\end{align*}
In particular, $\langle 	\boldsymbol{x},	\boldsymbol{u}\rangle_{\sim}=0$ if and only if $	\psi(\langle 	\boldsymbol{x},	\boldsymbol{u}\rangle_{\sim})=\boldsymbol{0}$, or equivalently,
\[\sum_{b\in P} x_{b,j}\overline{u_{b,j}}=0 \text{ for all } j=1,2,\dots, t_{\IIp} ~~\text{
and }~~\sum_{b\in P} y_{b,k} {v_{b,k}^\prime}=0=\sum_{b\in P}y_{b,k}^\prime {v_{b,k}} \text{ for all } k=1,2,\dots, t_{\IIIp}. \]
\end{proposition}

Using  the orthogonality in Proposition~\ref{orthogonality}, the Hermitian dual of $\mathcal{C}$ in (\ref{H3}) can be viewed of the form

\begin{align} 
	\mathcal{C}^{\perp_{\rm H}}\cong  \left(\prod_{j=1}^{t_{\IIp}} E_j ^{\perp_{\rm H}} \right)\times \left(\prod_{k=1}^{t_{\IIIp}} \left( (F_k^\prime) ^{\perp_{\rm E}}\times  F_k^{\perp_{\rm E}}\right) \right).
\end{align}

\section*{Acknowledgments} 
S. Jitman was supported   by the Thailand Research Fund and the Office of Higher Education Commission  of Thailand under Research
Grant MRG6080012. 



\begin{thebibliography}{800}
\bibitem{BGG2012} A. Batoul, K. Guenda, T. A. Gulliver,  On self-dual cyclic codes over finite chain rings, 	
	{Des. Codes Cryptogr.}   70 (2014)  347--358. 

 
	
\bibitem{B1997} S. Benson,  Students ask the darnedest things: A result in elementary group theory, 
	{Math. Mag.}  70  (1997)    207--211.

    \bibitem{BSG2016} C. Bocong,  S. Ling,  Z.  Guanghui,  Enumeration formulas for self-dual cyclic codes, {Finite Fields Appl.}  {42} (2016) 1--22.

\bibitem{BJU2018} A. Boripan, S. Jitman, P.  Udomkavanich,  Characterization and enumeration of complementary dual abelian codes, 
{J. Appl. Math. Comput.} {58} (2018) 527--544.

\bibitem{BJU2018}  A Boripan, {S. Jitman}, P. Udomkavanich, 
Self-conjugate-reciprocal irreducible monic factors of $x^n -1$ over finite fields and their applications, {Finite Fields Appl.}  {55}   (2019)  78--96.






\bibitem{CLZ2006} J. Chen, Y. Li, Y. Zhou,  Morphic group rings, 
	{J. Pure Appl. Algebra}    205 (2006)   621--639.
	
\bibitem{DL2004}  H. Dinh, S. R. L\'{o}pez-Permouth,  Cyclic and negacyclic codes over finite chain rings, 
	{IEEE Trans. Inform. Theory}  50 (2004)  1728--1744.   

\bibitem{D2007} T.  J.  Dorsey,  Morphic and principal-ideal group rings, 
  	{J. Algebra}   318  (2007)   393--411.


\bibitem{FiSe1976} J. L.  Fisher  S. K.  Sehgal,
	 Principal ideal group rings, 
	{Comm. Algebra}   4 (1976)   319--325.  
    


\bibitem{HKCSS1994}  A. R. Hammons Jr., P.V. Kumar, A. R. Calderbank, N. J. A. Sloane, P. Sol\'{e}, 
    The $\mathbb{Z}_4$ linearity of Kerdock, Preparata, Goethals and related codes, 
    {IEEE Trans. Inform. Theory}  40 (1994)   301--319.






\bibitem{JLX2011} Y. Jia, S. Ling, C. Xing,  On self-dual cyclic codes over finite fields, 
	{IEEE Trans. Inform. Theory}   57  (2011)  2243--2251.  

\bibitem{J2018}  S. Jitman, Good integers and some applications in coding theory, {Cryptogr. Commun.}  {10}  (2018) 685--704 and S. Jitman, Correction to: Good integers and some applications in coding theory, {Cryptogr. Commun.}  {10} (2018)  1203--1203.

\bibitem{JLLX2012}   S, Jitman, S. Ling, H. Liu,  X.  Xie,   Abelian codes in principal ideal group algebras,  
    {IEEE Trans. Inform. Theory}   59 (2013)   3046--3058.

\bibitem{JLS2013}   S, Jitman, S. Ling, E.  Sangwisut,    On self-dual cyclic codes of length $p^a$ over $\mathrm{ GR}({p^2},s)$, {Adv. Math. Commun} {10} (2016) 255--273.


\bibitem{JLS2012}   S, Jitman, S. Ling, P.  Sol\'{e},   Hermitian self-dual abelian codes,    {IEEE Trans. Inform. Theory}   60 (2014)    1496--1507.




\bibitem{KLL2008} H. M. Kiah, K. H. Leung, S. Ling,  Cyclic codes over ${\rm GR}(p^2,m)$ of length $p^k$,  
    {Finite Fields Appl.}    14 (2008)   834--846.  

\bibitem{KLL2012} H. M. Kiah, K. H. Leung, S. Ling,	 A note on cyclic codes over ${\rm GR}(p^2,m)$ of length $p^k$, 
    {Des. Codes Cryptogr.}  63 (2012)  105--112.  

\bibitem{KR2003} T. Kiran, B. S. Rajan,  Abelian codes over Galois rings closed under certain permutations,   
    {IEEE Trans. Inform. Theory}  49 (2003)   2242--2253.


    
\bibitem{M1997}  P.  Moree, On the divisors of $a^k+b^k$, {Acta Arithmetica} {LXXX}   (1997)  197--212.

	\bibitem{NRS2006} G. Nebe, E. M. Rains, N. J. A. Sloane, {Self-Dual Codes and
    Invariant Theory,} Algorithms and Computation in Mathematics vol. 17.
Berlin, Heidelberg: Springer-Verlag, 2006.

   
\bibitem{N1972}  W. K. Nicholson,   {Local group rings},  {Canad. Math. Bull.}  {15},(1972) 137--138.




\bibitem{S2006} A. S\u{a}l\u{a}gean,  Repeated-root cyclic and negacyclic codes over a finite chain ring,  { Discrete Appl.  Math.}  154 (2006) 413--419.

\bibitem{SE2009} R. Sobhani, M. Esmaeili,  A note on cyclic codes over ${\rm GR}(p^2,m)$ of length $p^k$,  
   {Finite Fields Appl.}    15 (2009)  387--391.
 
\bibitem{W2003} Z. X. Wan, {Lectures on Finite Fields and Galois Rings,}
World Scientific, New Jersey,  2003.


\bibitem{W2002} W. Willems,  A note on self-dual group codes, 
   {IEEE Trans. Inform. Theory}    48  (2002)    3107--3109.

 \bibitem{YM1994} X. Yang, J. L.  Massey,  The condition for a cyclic code to have a complementary dual, Discrete Math. {126} (1994)  391--393. 

	
\end{thebibliography}
\end{document}